%% file: phan_minmax.tex
\documentclass{siamart1116}

\input{ex_shared_fini}

\ifpdf
\hypersetup{
  pdftitle={\TheTitle},
  pdfauthor={\TheAuthors}
}
\fi

\begin{document}

\maketitle

\begin{abstract}
We search saddle points for a large class of convex-concave Lagrangian. A generalized explicit iterative scheme based on Arrow-Hurwicz method converges to a saddle point of the problem. We also propose in this work, a convergent semi-implicit scheme in order to accelerate the convergence of the iterative process. Numerical experiments are provided for a nontrivial  numerical problem modeling an optimal shape problem of thin torsion rods \cite{alibert2013nonstandard}. This semi-implicit scheme is figured out in practice robustly efficient in comparison with the explicit one.
\end{abstract}
\begin{keywords}
Saddle point, Lagrangian, Primal-dual algorithm.
\end{keywords}

\begin{AMS}
  49M29, 65B99, 65K05.
\end{AMS}

\section{Introduction}
In calculus of variations, we usually deal with problems of the following kind
$$
 \inf_{u\in C}  \int_\Omega [\varphi(\nabla u) + h(u)]\, dx \leqno(\Ph)
$$
where $\Omega$ is an open bounded subset of $\RR^N$, the constraint is given by $$C= \left\{ u\in W^{1,p}(\Omega), u=u_0 \tinm{on} \partial \Omega  \right\},$$
and functions $\varphi:\RR^N\to\RR$ and $h:\RR\to\RR$ are convex continuous. The integrands satisfy suitable growth conditions to ensure the well-posedness of the problem.
\medskip

Let us consider for instance a problem arising in shape optimization of thin torsion rods \cite{alibert2013nonstandard} to be a prototype situation
\begin{align*}
m(s):= \inf \left\{ \int_\Omega \varphi(\nabla u) : u\in H^1_0(\Omega), \int_\Omega u = s \right\}
\end{align*}
with $N=2$ and $\varphi$ being a convex function given by
\begin{align*}
\varphi(z) := 
\begin{cases}
\frac{1}{2}(|z|^2 +1) &\tinm{if} |z| \ge 1 \\
|z| &\tinm{if} |z| \le 1.
\end{cases}
\end{align*}
This is a convex problem without presence of $h$. Moreover, $\varphi$ is not differentiable at the origin.

\medskip
By duality argument, we can rewrite problem $\Pbh$ as  a saddle point problem
$$
\inf_{u\in C}\sup_{p\in K}  \int_\Omega [\nabla u \cdot p + h(u) - \varphi^*(p)]dx \leqno(\Mh)
$$
with $K= L^{p'}(\Omega;\RR^N)$ and $\varphi^*$ being the Fenchel conjugate of $\varphi$ \cite{ekeland1974analyse}. We denote by $L(u,p)$ the Lagrangian
$$ L(u,p) =\int_\Omega [\nabla u \cdot p + h(u) - \varphi^*(p)]dx.$$
Our aim is to find solutions of $\Pbh$ by means of the saddle point problem $\Mbh$. We recall that a saddle point $(\hat u,\hat p)$ of $L(u,p)$ in $C\times K$  is characterized by the inequalities
\begin{align}
 L(\hat u, p) \le L(\hat u,\hat p) \le L(u,\hat p),\ \forall u\in C,\ \forall p\in K.\label{chaUP}
\end{align}
Usually, we immediately think of the gradient descent-ascent so as to seek a saddle point. For a general Lagrangian $L(u,p)$, the simplest approach introduced by Arrow and Hurwicz has the form
\begin{align*}
&p_{n+1} = \Pi_K \Big( p_n +\tau_n \frac{\partial L}{\partial p}(u_n,p_n) \Big)\\
&u_{n+1} = \Pi_C \Big(u_n - \tau_n \frac{\partial L}{\partial u}(u_n,p_n) \Big).
\end{align*}
where $\Pi_K, \Pi_C$ are orthogonal projections on the closed convex sets $K$ and $C$, respectively.
However, this first-order iterative optimization algorithm
converges under very stringent conditions (like strict convexity-concavity) and special choosing of stepsizes $\tau_n \to 0,\; \sum_{n=0}^\infty \tau_n = \infty$ \cite{kallio1994perturbation}. To overcome these difficulties, L. D. Popov \cite{popov1980modification} gave a modification of the Arrow-Hurwicz method by introducing the so-called "leading" point, denoted by $ (\ov u_n, \ov p_n)$, which is an auxiliary point in order to jump to the next approximation with the help of the gradient direction,
\begin{align*}
&p_{n+1} = \Pi_K \Big( p_n +\tau \frac{\partial L}{\partial p}(\ov u_n, \ov p_n) \Big)\\
&u_{n+1} = \Pi_C \Big(u_n - \tau \frac{\partial L}{\partial u}(\ov u_n, \ov p_n) \Big)\\
&\ov p_{n+1} = \Pi_K \Big( p_{n+1} +\tau \frac{\partial L}{\partial p}(\ov u_n, \ov p_n) \Big)\\
&\ov u_{n+1} = \Pi_C \Big(u_{n+1} - \tau \frac{\partial L}{\partial u}(\ov u_n, \ov p_n) \Big).
\end{align*}
In his paper, he proved that there exists a positive scalar $\tau_0$ such that the modified algorithm converges for all constant stepsize $\tau$ taken in the interval $0<\tau < \tau_0$. This improvement enlarges the class of applicable problems whose convex-concave Lagrangians have derivatives satisfying Lipschitz conditions. It is clear that leading points makes the iterative processes more stable. But, because of extra projections, the more complicated the projections are, the heavier computation is. So, in some cases, replacing the extra projections as well as leading points is necessary. Chambolle-Pock et al. \cite{pock2009algorithm,chambolle2011first} started dealing with a typical Lagrangian which is a linear form
\begin{align*}
L(u,p) = \langle Au, p\rangle + \langle f, u \rangle - \langle g, p \rangle,
\end{align*}
where $A$ is a bounded linear operator and $\langle\cdot,\cdot \rangle$ denotes the associated scalar product on Hilbert spaces of variables $u,p$.
With these settings, they used simply computed leading points:\; $\ov p_{n+1} = p_{n+1}$ and $\ov u_{n+1} = 2u_{n+1} -u_n$. The replaced leading points are just a linear extrapolation based on the current and previous iterates. And, it is proved that the iterative process
\begin{align}
\begin{aligned}
& p_{n+1} = \Pi_K ( p_n +\alpha (A\ov u_n - g))\\
& u_{n+1} = \Pi_C (u_n - \beta (A^* p_{n+1} + f))\\
& \ov u_{n+1} = 2u_{n+1} -u_n
\end{aligned} \label{lAlgo}
\end{align}
converges to a saddle point of $L(u,p)$ if we choose $\alpha,\beta >0$ such that $\alpha\beta\| A\|^2 <1$.  Here, $A^*$ denotes the adjoint of operator  $A$. We remark that the algorithm (\ref{lAlgo}) without the projectors $ \Pi_K$ and $ \Pi_C$, can be interpreted as the Arrow-Hurwicz  algorithm for the augmented Lagrangian
$$
L_\beta(u,p)=L(u,p)-\beta \langle A^* p , A^* p \rangle
$$
and the augmented parameter $\beta$ is optimal for the convergence.

After that, there are many efforts to accelerate the convergence of algorithms of Arrow-Hurwicz type, as shown in the same paper \cite{chambolle2011first}, such as with varied stepsizes, modified the extrapolation of leading points and  implicit schemes.  In a more recent paper \cite{chambolle2016ergodic}, metric changes allow to enlarge stepsize. Such results are obtained for a general Lagrangian of type
\begin{equation}\label{L} 
L(u,p)=\langle Au, p\rangle + F(u) - G(p).
\end{equation}
Furthermore, numerous results on the convergence rates are achieved. Basically, the main idea is to handle the proximity technique with implicit schemes. But, the resolvent of the proximity should be easy to compute in practice. The most important is that the projections are completely ignored in almost proofs by setting $C, K$ Hilbert spaces. To the best of our knowledge, treating the Lagrangian of kind \eqref{L} with non trivial projections is still missing.

\medskip
Problem $\Pbh$ can be modeled for free boundary problems \cite{AlCa1981free}, two-phase problem of Cahn-Hilliard fluid \cite{bouchitte1994cahn}, this kind of problem also often appears in computer vision as imaging problems \cite{chambolle2011first} where $\varphi$ stands for regularizers and $h$ for dataterms. But, in almost cases, $h$ owns non convexity. The interest is to find global solutions among local ones of problem $\Pbh$. In that circumstance, we shall need some convexification recipe. By duality theory, fortunately, convexification procedure \cite {BouFra2015duality,BouFra2016duality,BouPha2017duality} often gives convex representations of $\Pbh$ under a min-max form of $L(u,p)$ in \eqref{L}.  We hence restrict to the case $F$, $G$ being convex for simplification.

\medskip
In this paper, we deal with the general Lagrangian \eqref{L}  with non trivial projections on closed convex sets $C,K$ of Hilbert spaces. We then, arrive to generalize the scheme \eqref{lAlgo} to general $F,G$ being differentiable, namely with explicit scheme leading to simple implementation. In Section \ref{escheme}, we prove the convergence of this explicit scheme. In the next section, we propose a semi-implicit scheme which is very robust in comparison with the fully explicit scheme since the numerical parameters $\alpha,\beta$ do not depend on $\lVert A \rVert$. The following section deals with numerical results and exhibits the advantage of such a semi-implicit algorithm, namely for a variant with a splitting technique  reducing the number iteration of implicit solvers. The computational cost of the different algorithms are then compared and shows the interest of the proposed algorithm.
The last section is intended for conclusion.

\bigskip\noindent
 {\bf Acknowledgments.} 
We are very grateful to G. Bouchitt\'e and I. Fragal\`a  for fruitful discussions and collaborations on the topic.
\bigskip

\section{Saddle point problem and explicit scheme \label{escheme}}
Let $C$ and $K$ be closed convex non empty subsets of Hilbert spaces $V$ and $W$, respectively. We denote by $\langle \cdot, \cdot \rangle$ the inner product and by $\lVert \cdot\rVert:= \sqrt{\langle\cdot,\cdot\rangle}$ the corresponding norm on both Hilbert spaces without ambiguity. Given $A: V \to W$ a continuous linear operator with its induced norm
\begin{align*}
\lVert A \rVert := \sup \left\{ \lVert Av\rVert : v\in V, \lVert x \rVert \le 1 \right\}.
\end{align*}
We consider an inf-sup optimization problem in the very generic form 
\begin{align}
\inf_{u\in C} \sup_{p\in K} L(u,p) \quad \tinm{with} \quad L(u,p)=\langle Au, p\rangle + F(u) - G(p),\label{Pb}
\end{align}
where $F$, $G$ are convex functions and supposed to be differentiable. Their derivatives satisfy the Lipschitz condition with constants $L_f$, $L_g$, respectively. We assume that the set of saddle points $\hat C\times \hat K$ of Lagrangian $L(u,p)$ is not empty. 

\medskip
The aim is to find a saddle point $(\hat u,\hat p)$ of $L(u,p)$ in $C\times K$. 
Now, let us generalize the explicit scheme introduced in \eqref{lAlgo} for the general saddle point problem \eqref{Pb}. Basically, we keep the main idea of the convergence proof by Chambolle-Pock et al. \cite{pock2009algorithm} with additional technical difficulties due to additional convex function $F$ and $G$.

\medskip
For initialization, we choose $\alpha, \beta >0$, $(u_0,p_0)\in C\times K$, $\ov u_0 = u_0$. We propose an iterative algorithm as below
\begin{align}
\begin{cases}
p_{n+1} = \Pi_K ( p_n +\alpha (A\ov u_n - G'(p_n)))\\
u_{n+1} = \Pi_C (u_n - \beta (A^* p_{n+1} + F'(u_n)))\\
\ov u_{n+1} = 2u_{n+1} -u_n
\end{cases}\label{AlgG}
\end{align}
where $A^*$ stands for the adjoint of operator  $A$ and $\Pi_K$, $\Pi_C$ respectively denote the orthogonal projectors on closed convex sets $K$, $C$.\\
We get below the convergence result:
\bigskip

\begin{theorem}
\label{Eth}
Under the standing assumption, for all $\alpha,\beta$ such that 
\begin{align}
\begin{aligned}
0<\alpha <\frac{2}{L_g}, \quad & \quad
0 <\beta < \frac{2}{L_f}, \\
\alpha \beta \left(\lVert A \rVert^2 - \frac{L_f L_g}{4} \right) &+ \frac{\alpha L_g}{2} + \frac{\beta L_f}{2} < 1,
\end{aligned}
\label{ab}
\end{align}
the proposed algorithm \eqref{AlgG} converges to a saddle point of $L(u,p)$ in the set $C\times K$.
\end{theorem}

\bigskip
Before proving the theorem, let us recall some important properties of orthogonal projection on a closed convex set.
\begin{proposition}
Let $\Pi_D$ be an orthogonal projection on a closed convex subset $D$ in a Hilbert space $V$. Followings hold true: 
\begin{itemize}
\item[(i)] For every $u,v\in V$,
\begin{align}
\langle \Pi_D(u) - \Pi_D(v), u - v \rangle \ge \Vert \Pi_D(u) - \Pi_D(v) \rVert ^2.
\label{Pmono}
\end{align}
As a consequence, the projection is a monotone $1$-Lipschitz operator.
\item[(ii)] For every $z \in D$,
\begin{align*}
\lVert u - \Pi_D(u) \rVert^2 + \lVert z - \Pi_D(u) \rVert^2 \le \lVert u - z \rVert^2. 
\end{align*}
In particular, if $0\in D$ then
\begin{align}
\lVert u - \Pi_D(u) \rVert^2 + \lVert \Pi_D(u) \rVert^2 \le \lVert u  \rVert^2. \label{Pythagorean}
\end{align}
\item[(iii)] For every $v\in V$,
\begin{align}
\Pi_D(u)  -v = \Pi_{D-v}(u-v). \label{trPi}
\end{align}
where $D-v$ is the Minkowski addition, i.e. $D-v := \{ d-v: d\in D \}$.
\end{itemize}
\end{proposition}
\begin{proof}
(i) For every $u\in V$, the characterization of $\Pi_D(u)$ is given by
\begin{align}
\langle u - \Pi_D(u), z- \Pi_D(u) \rangle \le 0 \quad  \forall z \in D\label{ChPi}
\end{align}
We choose $z =\Pi_D(v)$, then
\begin{align}
\langle u - \Pi_D(u), \Pi_D(v)- \Pi_D(u) \rangle \le 0.\label{ineq2}
\end{align}
Similarly, we have
\begin{align}
\langle v - \Pi_D(v), \Pi_D(u)- \Pi_D(v) \rangle \le 0.
\label{ineq3}
\end{align}
Summing inequalities \eqref{ineq2} and \eqref{ineq3}, we obtain the inequality \eqref{Pmono}. The monotonicity and Lipschitz continuity are direct consequences.\par
(ii) For every $z\in D$, by using \eqref{ChPi}, we have
\begin{align*}
\langle u- z, u- z \rangle 
=& \lVert u- \Pi_D(u) \rVert^2 + \lVert z - \Pi_D \rVert^2 + 2 \langle u- \Pi_D(u), \Pi_D(u) -z \rangle \\
\ge & \lVert u- \Pi_D(u) \rVert^2 + \lVert z - \Pi_D \rVert^2.
\end{align*}

(iii) For every $z\in D$, we recall the characterization \eqref{ChPi},
\begin{align*}
\langle u - \Pi_D(u), z- \Pi_D(u) \rangle \le 0 .
\end{align*}
This is equivalent to
\begin{align*}
\langle u-v - (\Pi_D(u) -v), z-v - (\Pi_D(u)-v) \rangle \le 0 ,\quad \forall z\in D.
\end{align*}
It is to say that $\Pi_D(u) -v$ is the projection of $(u-v)$ on the set $D-v$.
\end{proof}

\bigskip

\begin{proof}[Proof of Theorem \ref{Eth}]
Let $(\hat u,\hat p)\in \hat C\times \hat K$ be a saddle point of $L(u,p)$.
Beside the characterization given by the inequalities \eqref{chaUP}, saddle points of the problem \eqref{Pb} can be characterized by, see \cite{ekeland1974analyse} in detail: 

$(\hat u, \hat p)$ is a saddle point of $L(u,p)$ in $C\times K$ if and only if
\begin{align}
\langle A^*\hat p + F'(\hat u), u - \hat u \rangle \ge 0 \; \tinm{for all} \; u\in C \label{chaU}\\
\langle A\hat u - G'(\hat p), p -\hat p \rangle \le 0 \; \tinm{for all}\; p\in K. \label{chaP}
\end{align}
Under the standing assumption, $F$ and $G$ have second derivatives defined almost everywhere. Then, we obtain
\begin{align}
F'(u) - F'(v) &= \int_0^1 F''((1-\theta)v + \theta u))(u - v)\, d\theta \quad \forall u, v \in V, \label{dF2}\\
G'(p) - G'(q) &= \int_0^1 G''((1-\theta)q + \theta p))(p - q)\, d\theta \quad \forall p,q \in W. \label{dG2}
\end{align}
We define
\begin{align*}
f_{v,u} := \int_0^1 F''((1-\theta)v + \theta u))d\theta , \qquad \qquad
g_{q,p} := \int_0^1 G''((1-\theta)q + \theta p))d\theta.
\end{align*}
We observe that $f_{v,u}$, $g_{q,p}$ are symmetric bilinear forms on product spaces $V\times V$ and $W\times W$, respectively.
Since $F$, $G$ are convex functions, $f_{v,u}$ and $g_{q,p}$ are  positive semi-definite. 
Such bilinear forms admit factorizations by operators, see \cite{rudin1991functional} Theorem 12.33 p.331, as 
$$
\langle f_{v,u} V,V \rangle=\langle M_{f_{v,u}} V,M_{f_{v,u}} V \rangle, \qquad \qquad \langle g_{q,p} W, W \rangle=\langle M_{g_{q,p}} W,M_{g_{q,p}} W \rangle.
$$
Furthermore, by Lipschitz conditions, we have
\begin{align*}
\langle f_{v,u} v',u' \rangle \le L_f\lVert v' \rVert\lVert u' \rVert \quad \forall u', v' \in V \, \tinm{ and } \,
 \langle g_{q,p} q',p' \rangle \le L_g\lVert q' \rVert\lVert p' \rVert \quad \forall p', q' \in W.
\end{align*}
Hence, $f_{v,u}$, $g_{q,p}$ are continuous bilinear forms, i.e. $f_{v,u}\in \LL(V\times V;\RR)$, $g_{q,p}\in \LL(W\times W;\RR)$.
In what follows,  the role of symmetric positive semi-definite bilinear forms $f_{v,u}$, $g_{q,p}$ is exploited.


We firstly introduce some useful notation:
\begin{align}
\begin{aligned}
&P_n = p_n -\hat p, \quad
U_n = u_n -\hat u, \quad
\ov U_n = \ov u_n - \hat u, \quad
f_n = f_{\hat u, u_n} \, ,\quad
g_n = g_{\hat p, p_n}\, ,\\
&\widetilde K = K - \hat p, \qquad \widetilde C = C- \hat u.
\end{aligned}
\label{nota1}
\end{align}
By using \eqref{trPi} with the settings \eqref{nota1} above, we can rewrite the iterative process \eqref{AlgG} as
\begin{align*}
\begin{cases}
P_{n+1} = \Pi_{\widetilde K} (P_n +\alpha (A\ov u_n - G'(p_n)))\\
U_{n+1} = \Pi_{\widetilde C} (U_n - \beta (A^* p_{n+1} + F'(u_n )))\\
\ov u_{n+1} = 2u_{n+1} -u_n.
\end{cases}
\end{align*}
We notice that the couple of equations \eqref{dF2} - \eqref{dG2}  give us the following representation
\begin{align}
F'(u_n) - F'(\hat u) &= f_n U_n \label{df2}\\
G'(p_n) - G'(\hat p) &= g_n P_n. \label{dg2}
\end{align}
Since $0\in \widetilde K$, we can handle the inequality \eqref{Pythagorean} in order to deduce that
\begin{align*}
\lVert P_{n+1} \rVert ^2 
&\le \lVert P_n +\alpha (A\ov u_n - G'(p_n)) \rVert^2 - \lVert P_{n+1} -  P_n -\alpha (A\ov u_n - G'(p_n)) \rVert^2 \\
&= \lVert P_n \rVert^2 - \lVert P_{n+1} - P_n \rVert^2 + 2\alpha \langle A \ov u_n - G'(p_n), P_{n+1} \rangle \\
&\le \lVert P_n \rVert^2 - \lVert P_{n+1} - P_n \rVert^2 + 2\alpha \langle A \ov U_n - g_n P_n), P_{n+1} \rangle.
\end{align*}
The second inequality is obtained by adding a non negative amount $-\langle A\hat u - G'(\hat p), p -\hat p \rangle$, see \eqref{chaP} for evidence and passing the equality \eqref{dg2}.
Similarly, we just repeat the same procedure for the variable $u$
\begin{align*}
\lVert U_{n+1} \rVert ^2 
&\le \lVert U_n - \beta (A^* p_{n+1} + F'(u_n)) \rVert^2 - \lVert U_{n+1} -  U_n + \beta (A^* p_{n+1} + F'(u_n) ) \rVert^2 \\
&\le \lVert U_n \rVert^2 - \lVert U_{n+1} - U_n \rVert^2 - 2\beta \langle A^* P_{n+1} + f_n U_n, U_{n+1} \rangle.
\end{align*}
For short, let us denote
\begin{align*}
S_n &= \alpha \lVert U_n\rVert^2 + \beta \lVert P_n \rVert^2 , \\
R_n &= \alpha \lVert U_{n+1} - U_n \rVert^2 + \beta \lVert P_{n+1} - P_n \rVert^2.
\end{align*}
It then follows that
\begin{align}
S_{n+1} \le S_n -R_n -2\alpha\beta \left( \langle A^* P_{n+1} + f_n U_n, U_{n+1} \rangle - \langle A \ov U_n - g_n P_n, P_{n+1} \rangle \right). \label{Sne}
\end{align}
By definition of adjoint operator, $\langle A^* P_{n+1} ,U_{n+1}\rangle = \langle AU_{n+1}, P_{n+1}\rangle$, we observe that
\begin{align*}
&-2\alpha\beta \left[ \langle A^* P_{n+1} + f_n U_n, U_{n+1} \rangle - \langle A \ov U_n - g_n P_n, P_{n+1} \rangle \right]\\
=& -2\alpha\beta \left[ \langle  f_n U_n, U_{n+1} \rangle +\langle g_nP_n, P_{n+1}\rangle - \langle A (\ov U_n -  U_{n+1}), P_{n+1}\rangle \right].
\end{align*}
In the followings, let us do some necessary estimations. By using definition of leading point $\ov U_n$ and continuity of operator $A$, we deduce that
\begin{align*}
\langle A(U_{n+1} - \ov U_n), P_{n+1}\rangle 
=& \langle A(U_{n+1} - U_n), P_{n+1}\rangle - \langle A(U_n -  U_{n-1}), P_{n+1}\rangle \\
=& \langle A(U_{n+1} - U_n), P_{n+1}\rangle - \langle A(U_n -  U_{n-1}), P_n\rangle \\
&- \langle A(U_n -  U_{n-1}), P_{n+1} - P_n\rangle \\
\ge & \langle A(U_{n+1} - U_n), P_{n+1}\rangle - \langle A(U_n -  U_{n-1}), P_n\rangle \\
& - \lVert A \rVert \lVert U_n -U_{n-1} \rVert \lVert P_{n+1} - P_n \rVert.
\end{align*}
We have already mentioned that the bilinear form $f_n$ have positive square root $M_{f_n}$. By taking into account this factorization, we get
\begin{align}
\begin{aligned}
- \langle  f_n U_n, U_{n+1} \rangle 
&= -\langle M_{f_n}  U_{n+1}, M_{f_n} U_{n+1} \rangle -  \langle M_{f_n} (U_n - U_{n+1}), M_{f_n} U_{n+1} \rangle \\
&\le (-1+1) \lVert M_{f_n} U_{n+1} \rVert^2 + \frac{1}{4}\lVert M_{f_n} (U_{n+1} - U_n) \rVert^2\\
&\le \frac{L_f}{4}\lVert U_{n+1} - U_n \rVert^2,
\end{aligned}\label{estimF}
\end{align}
and analogously,
\begin{align}
- \langle  g_n P_n, P_{n+1} \rangle 
\le  \frac{L_g}{4}\lVert P_{n+1} - P_n \rVert^2.
\label{estimG}
\end{align}
We recall that the inequality $2ab \le (\delta a^2 + b^2/\delta)$ holds true for any $a,b$ and any $\delta >0$.
We now make use of the previous estimations. For any $\delta >0$, it holds
\begin{align*}
S_{n+1}
\le  S_n & -\alpha \left(1- \frac{\beta L_f}{2} \right) \lVert U_{n+1} - U_n \rVert^2 - \beta \left(1- \frac{\alpha L_g}{2} \right)\lVert P_{n+1} - P_n \rVert^2 \\
& -2\alpha\beta \left[ \langle A(U_{n+1} - U_n), P_{n+1}\rangle - \langle A(U_n -  U_{n-1}), P_n\rangle \right] \\
&+ 2\alpha\beta  \lVert A \rVert \Vert U_n -U_{n-1} \rVert \lVert P_{n+1} - P_n \rVert \\
\le  S_n & -\alpha \left(1- \frac{\beta L_f}{2} \right) \lVert U_{n+1} - U_n \rVert^2 - \beta \left(1- \frac{\alpha L_g}{2} \right)\lVert P_{n+1} - P_n \rVert^2 \\
& -2\alpha\beta \left[ \langle A(U_{n+1} - U_n), P_{n+1}\rangle - \langle A(U_n -  U_{n-1}), P_n\rangle \right] \\
&+ \alpha\beta  \lVert A \rVert \left( \delta\Vert U_n -U_{n-1} \rVert^2 + \frac{1}{\delta}\lVert P_{n+1} - P_n \rVert^2 \right).
\end{align*}
Given $M\ge N \ge 1$ and $\gamma >0$, let us take the sum of the inequality \eqref{Sne} from $N$ up to $M$:
\begin{align*}
S_{M+1} \le S_{N} & -\alpha \left(1- \frac{\beta L_f}{2} \right)\sum_{n=N}^{M} \lVert U_{n+1} - U_n \rVert^2 - \beta \left(1- \frac{\alpha L_g}{2} \right) \sum_{n=N}^{M} \lVert P_{n+1} - P_n \rVert^2 \\
& -2\alpha\beta \left[ \langle A(U_{M+1} - U_M), P_{M+1}\rangle - \langle A(U_N -  U_{N-1}), P_N\rangle \right] \\
&+ \alpha\beta  \lVert A \rVert \delta \sum_{n=N-1}^{M-1} \Vert U_{n+1} -U_n \rVert^2 + \frac{\alpha\beta  \lVert A \rVert}{\delta} \sum_{n=N}^{M} \lVert P_{n+1} - P_n \rVert^2\\
\le S_{N} & -\alpha \left(1- \frac{\beta L_f}{2} -\beta \lVert A \rVert \delta \right)\sum_{n=N}^{M-1} \lVert U_{n+1} - U_n \rVert^2\\
& - \beta \left(1- \frac{\alpha L_g}{2} -\frac{\alpha \lVert A\rVert}{\delta} \right) \sum_{n=N}^{M} \lVert P_{n+1} - P_n \rVert^2 \\
&-\alpha \left(1- \frac{\beta L_f}{2} \right) \lVert U_{M+1} - U_M \rVert^2 +  \alpha\beta  \lVert A \rVert \delta \Vert U_N -U_{N-1} \rVert^2 \\
& +2\alpha\beta \langle A(U_N -  U_{N-1}), P_N\rangle \\ & +\alpha\beta  \lVert A \rVert \left( \gamma\Vert U_{M+1} -U_{M} \rVert^2 + \frac{1}{\gamma}\lVert P_{M+1} \rVert^2 \right).
\end{align*}
We develop $S_{M+1}$ and rearrange the previous calculation to obtain 
\begin{align}
\begin{aligned}
\alpha \lVert U_{M+1} \rVert^2 
&+ \beta \left(1- \frac{\alpha \lVert A\lVert}{\gamma} \right) \lVert P_{M+1} \rVert^2 \\ 
&+\alpha \left(1- \frac{\beta L_f}{2} -\beta \lVert A \rVert \delta \right)\sum_{n=N}^{M-1} \lVert U_{n+1} - U_n \rVert^2\\
&+ \beta \left(1- \frac{\alpha L_g}{2} -\frac{\alpha \lVert A\rVert}{\delta} \right) \sum_{n=N}^{M} \lVert P_{n+1} - P_n \rVert^2 \\
&+\alpha \left(1- \frac{\beta L_f}{2} -\beta  \lVert A \rVert \gamma \right)\lVert U_{M+1} - U_M \rVert^2 \\
\le S_{N} &
+  \alpha\beta  \lVert A \rVert \delta \Vert U_N -U_{N-1} \rVert^2 +2\alpha\beta \langle A(U_N -  U_{N-1}), P_N\rangle.
\end{aligned}
\label{mainconv}
\end{align}
We observe that for $\alpha, \beta, \delta, \gamma$ being positive numbers, it holds
\begin{align}
\begin{cases}
\beta \left(1- \frac{\alpha \lVert A\lVert}{\gamma} \right) > 0 \\
\alpha \left(1- \frac{\beta L_f}{2} -\beta \lVert A \rVert \delta \right) > 0 \\
\beta \left(1- \frac{\alpha L_g}{2} -\frac{\alpha \lVert A\rVert}{\delta} \right) > 0 \\
\alpha \left(1- \frac{\beta L_f}{2} -\beta  \lVert A \rVert \gamma \right) > 0
\end{cases}
\qquad
\Longleftrightarrow
\qquad
\begin{cases}
\alpha \lVert A\lVert < \gamma < \frac{2-\beta L_f}{2\beta \lVert A\rVert}\\
\frac{2\beta \lVert A\rVert }{2- \beta L_f} < \frac{1}{\delta} < \frac{2-\alpha L_g}{2\alpha \lVert A\rVert}.
\end{cases}\label{IntConv}
\end{align}
We then derive from \eqref{IntConv} that
\begin{align}
\begin{aligned}
0<\alpha <\frac{2}{L_g}, \quad & \quad
0 <\beta < \frac{2}{L_f}, \\
\alpha \beta \left(\lVert A \rVert^2 - \frac{L_f L_g}{4} \right) &+ \frac{\alpha L_g}{2} + \frac{\beta L_f}{2} < 1.
\end{aligned}\label{paramConv}
\end{align}
Therefore, if we choose $(\alpha,\beta)$ satisfying \eqref{paramConv} and $(\delta,\gamma)$ given in the intervals defined by \eqref{IntConv}, the left-hand side of inequality \eqref{mainconv} is positive. We see that two sequences $\lVert U_{M+1} \rVert$ and $\lVert P_{M+1}\rVert$ are bounded while both $\lVert U_{M+1} - U_M \rVert$ and $\lVert P_{M+1} - P_M \rVert$ converge to 0 as $M\to +\infty$. Or, equivalently, $\{u_M\}$ and $\{p_M\}$ are bounded and the sequences $\lVert u_{M+1} - u_M \rVert$, $\lVert p_{M+1} - p_M \rVert$ go to 0 as $M\to +\infty$. So, there exists a subsequence $(u_{M_k},p_{M_k})$ weakly converging to some $(u_*,p_*)\in C\times K$, furthermore, $u_{M_k + 1}$ and $\ov u_{M_k}$ converge to $u_*$ whilst $p_{M_k +1}$ converges to $p_*$. By passing in the limit in \eqref{AlgG}, we have
\begin{align*}
p_* & = \Pi_K ( p_* +\alpha (A u_* - G'(p_*)),\\
u_* & = \Pi_C (u_* - \beta (A^* p_* + F'(u_*)).
\end{align*}
This shows that $(u_*,p_*)$ is solution to problem \eqref{Pb}. We now can replace $\hat u = u_*$, $\hat p = p_*$ and $N=M_k$ in \eqref{mainconv}. Then, as $k$ is large enough, the right-hand side of \eqref{mainconv} will arbitrarily small. Thus, for every $M > M_k$, $\lVert U_M \rVert$ and $\lVert P_M \rVert$ are as small as we want. We conclude that $(U_M,P_M)$ converges to $(0,0)$ as $M\to +\infty$.
\end{proof}

\begin{remark}
If $F=G=0$,  the problem \eqref{Pb} reduces to
\begin{align}
\inf_{u\in C} \sup_{p\in K} \langle Au, p\rangle ,
\label{PbS}
\end{align}
and the algorithm \eqref{AlgG} becomes
\begin{align}
\begin{cases}
p_{n+1} = \Pi_K ( p_n +\alpha A\ov u_n )\\
u_{n+1} = \Pi_C (u_n - \beta A^* p_{n+1})\\
\ov u_{n+1} = 2u_{n+1} -u_n.
\end{cases}\label{AlgoS}
\end{align}
In that case, positive parameters $\alpha, \beta$ will be chosen such that $\alpha\beta\lVert A\rVert^2 <1$, so that the process \eqref{AlgoS} converges to a saddle point of $L(u,p)$ in $C\times K$. The choice of $\alpha,\beta$ is now more flexible, without upper bounds of Lipschitz constants.
\end{remark}

\bigskip
\section{Semi-implicit scheme} 
In the explicit scheme based on Arrow - Hurwicz method, the algorithm \eqref{AlgG} is convergent under the permanent appearance of the boundedness of the linear operator $A$. We wonder whether we find out an process which converges to a saddle point of $L(u,p)$ whose numerical parameters do not depend on the boundedness of $A$.  In the next works, we shall show up such a process, namely semi-implicit algorithm.

The idea of the following algorithm comes from the ascertainment that the steps $\alpha$ and $\beta$ of the algorithms \eqref{AlgG}  and \eqref{AlgoS}are limited by large eigenvalues  of the operator $A$ whereas this steps could be increased for the part of the iterate associated to low eigenvalues. We then aim for progress with optimal steps $\alpha$ and $\beta$ whatever the considered eigenmode of the iterate.

We suppose that $C \subset \dom(A)$. Problem \eqref{Pb} can be written as
\begin{align}
\min_{(x,q)\in \ddot C} \max_{y\in K} \langle q, p \rangle + F(u) - G(p)
\end{align}
with $\ddot C:=\{(u,Au): u\in C \}$. Now, if we apply the process \eqref{AlgG} to the problem in this form, we obtain
\begin{align}
\begin{cases}
p_{n+1} = \Pi_K ( p_n +\alpha (\ov q_n - G'(p_n)))\\
(u_{n+1},q_{n+1}) = \Pi_{\ddot C} (u_n - \beta F'(u_n), q_n - \beta p_{n+1})\\
\ov q_{n+1} = 2q_{n+1} -q_n
\end{cases}
\label{preSI}
\end{align}
It is evident that $\ddot C \subset V\times \im A \subset V\times W$. As $V\times \im A$ is a linear subspace of $V\times W$, we deduce that 
\begin{align*}
\Pi_{\ddot C}  (u_n - \beta F'(u_n), q_n - \beta p_{n+1})
= &\; \Pi_{\ddot C}(\Pi_{V\times \im A} (u_n - \beta F'(u_n), q_n - \beta p_{n+1})) \\
= & \; \Pi_{\ddot C} ((u_n,q_n) - \beta \,\Pi_{V\times \im A}(F'(u_n), p_{n+1}) ).
\end{align*}
The projection $\Pi_{V\times \im A}(u^0,q^0)$ is indeed to search an optimizer for the problem
\begin{align*}
\tinm{minimize}\quad u\mapsto \frac{1}{2} \left( \lVert Au - q^0 \rVert^2 + \lVert u - u^0 \rVert^2 \right).
\end{align*}
This is in fact a proximal operator of a quadratic form. Its resolvent is easily determined,
\begin{align*}
 u^* = \prox_{A,q^0}(u^0)= (A^*A + I)^{-1}(A^* q^0 + u^0).
\end{align*}
Besides, whenever operator $A$ is bounded, it holds
\begin{align*}
\lVert Au - A  u^* \rVert^2 + \lVert u -  u^* \rVert^2
\le (\lVert A\rVert^2 +1) \lVert u-  u^* \rVert^2.
\end{align*}
Then, the proximal operator assures the couple $(\Pi_C ( u^*), A \Pi_C ( u^*))$ isn't too far from $( u^*, A u^*)$.
This leads to the idea of replacing the projection $\Pi_{\ddot C}$ in the process \eqref{preSI} by a simpler approximation
\begin{align}
\begin{cases}
u_{n+1} = \Pi_C (u_n - \beta \widetilde u_n)\\
q_{n+1} = A u_{n+1}
\end{cases}
\label{preProj}
\end{align}
with $\widetilde u_n$ being the proximal point determined by
\begin{align}
\widetilde u_n = \prox_{A,p_{n+1}}(F'(u_n))= (A^*A + I)^{-1}(A^* p_{n+1} +F'(u_n)).\label{utilde}
\end{align}
Therefore, we introduce a semi-implicit scheme:
\begin{align}
\begin{cases}
p_{n+1} = \Pi_K ( p_n +\alpha (A\ov u_n - G'(p_n))\\
u_{n+1} = \Pi_C (u_n - \beta (A^*A + I)^{-1}(A^* p_{n+1} +F'(u_n)) )\\
\ov u_{n+1} = 2u_{n+1} -u_n
\end{cases}\label{AlgI}
\end{align}

\begin{remark}
If $C=V$, replacing the projection $\Pi_{\ddot C}$ in the process \eqref{preSI} by \eqref{preProj} is straightforward. Ortherwise, we notice that once $A$ is an isometric operator, i.e. $\lVert A u \rVert= \lVert u \rVert$, $\forall u\in \dom(A)$, this replacement is clearly equivalent. Then, in other words, the expression \eqref{preProj} defines a projector on $\ddot C$ which naturally coincides with the projector $\Pi_{\ddot C}$. In addition, if $A$ is an orthogonal operator, in this case, $A u_{n+1}$ is indeed the projection of $A (u_n - \beta \widetilde u_n)$ on the image $A C$, which will be shown in Lemma \ref{ImpProj}.
\end{remark}

\bigskip

\begin{lemma}
\label{ImpProj}
Let $O$ be a closed convex subset of Hilbert space $V$ such that $0\in O$ and $\Pi_O$ be the orthogonal  projector on $O$. For every densely defined, closed, linear operator $A: V \to W$ satisfying 
\begin{align}
\forall O,\quad \forall u \in V,\quad  \langle A(u - \Pi_O(u)), A \Pi_O(u) \rangle \ge 0,\label{Ha}
\end{align}
the process \eqref{preProj} is identified with
\begin{align}
\begin{cases}
u_{n+1} = \Pi_C (u_n - \beta \widetilde u_n)\\
q_{n+1} = \Pi_{A C}( A u_n - \beta A \widetilde u_n).
\end{cases}
\label{ProjP}
\end{align}
\end{lemma}

\begin{proof}
By making use of notation \eqref{trPi}, it is easy to verify that
\begin{align*}
u_{n+1} = \Pi_C (u_n - \beta \widetilde u_n) \quad \Longleftrightarrow  \quad u_{n+1} - u_0 = \Pi_{C-u_0} (u_n - u_0 - \beta \widetilde u_n), \; \forall u_0\in C.
\end{align*}
Since $0$ is always in $C-u_0$, and by dealing with the characterization of the projection $\Pi_{C-u_0}$, we derive that for all $u_0 \in C$,
\begin{align*}
\langle u_n - u_0 - \beta \widetilde u_n - (u_{n+1} - u_0), u_0 -  u_{n+1} \rangle \le 0.
\end{align*}
We deduce from the hypothesis \eqref{Ha} on the operator $A$ that
\begin{align}
\forall u_0 \in C, \quad \langle A u_n - \beta A \widetilde u_n - A u_{n+1}, A u_0 - A u_{n+1} \rangle \le 0.
\label{PiAc}
\end{align}
$AC$ is closed convex. {Convexity is preserved by linearity and closedness is preserved by closedness of $A$.}
The inequality \eqref{PiAc} well defines a projection on $AC$,
\begin{align*}
A u_{n+1} = \Pi_{A C} (A u_n - \beta A \widetilde u_n).
\end{align*}
If $q_{n+1}$ is a projection of $Au_n - \beta A\widetilde u_n$ on $A C$ then by the uniqueness of projection, $q_{n+1}$ must coincide with $Au_{n+1}$.
The proof is completed.
\end{proof}

\begin{remark}
When $A$ is an orthogonal operator, it evidently satisfies the hypothesis \eqref{Ha}. At the moment, keeping in mind that $A^*A=I$, it gives an identification
\begin{align*}
\langle A(u - \Pi_O(u)), A \Pi_O(u) \rangle = \langle A^*A(u - \Pi_O(u)), \Pi_O(u) \rangle =\langle u - \Pi_O(u), \Pi_O(u) \rangle.
\end{align*}
\end{remark}

\bigskip

In practice, $A$ usually stands for gradient operator $\nabla$. Let us show that, in this case, gradient operator satisfies the hypothesis \eqref{Ha}. In the  following lemma, we regard $u$ as a function of variable $x$ in a suitable space, for example $u\in V= L^2(\Omega)$. We say that $u$ has local constraints in the convex $C$ if $u(x)\in C(x)$ for all $x\in\Omega$.

\bigskip
\begin{lemma} Let $O$ be a closed convex subset of Hilbert space $V$ such that $0\in O$ and $\Pi_O$ be the orthogonal  projector on $O$. We suppose, in addition, that the projection $\Pi_O$ is local in the sense that $\mathring u = \Pi_O (u)$ is equivalent to $\mathring u(x) = \Pi_{O(x)} (u(x)), \forall x\in \Omega$.
Then, the gradient operator $\nabla$ satisfies the hypothesis \eqref{Ha}.
\end{lemma}

\begin{proof}
For every $u,v\in V$, let $\mathring u$, $\mathring v$ be the projections of $u,v$ on $O$, respectively. By using the inequality \eqref{Pmono}, we have that 
\begin{align*}
\langle u - v - (\mathring u- \mathring v), \mathring u - \mathring v \rangle \ge 0.
\end{align*}
Since the projection $\Pi_O$ is realized locally, choosing $v\in V$ such that $v(x)=u(x+th)$ for $t > 0$ and some given $h \in \RR^N$, we have
\begin{align*}
\frac{1}{t^2}
\langle u(x) - u(x +th) - [\mathring u(x) - \mathring u(x +th)], \mathring u(x) - \mathring u(x +th) \rangle \ge 0.
\end{align*}
We recall that the Gateaux derivative of $u$ is defined by
\begin{align*}
\langle \nabla u(x), h \rangle : = \lim_{t\to 0^+} \frac{u(x+th)-u(x)}{t}.
\end{align*}
By passing to limit as $t\to 0^+$, we obtain
\begin{align*}
\langle \nabla u(x)- \nabla\mathring u(x), \nabla\mathring u(x) \rangle \ge 0.
\end{align*}
In other words, we get
\begin{align*}
\langle \nabla u -  \nabla \Pi_O(u) , \nabla \Pi_O(u) \rangle \ge 0.
\end{align*}
This completes the proof of lemma.
\end{proof}

\bigskip
It is ready to prove the convergence of semi-implicit scheme proposed in \eqref{AlgI}. Here are the main result:

\bigskip
\begin{theorem}
Let $A: V \to W$ be a densely defined, closed linear operator satisfied the hypothesis \eqref{Ha}. For all $\alpha, \beta$ such that 
\begin{align*}
0<\alpha <\frac{2}{L_g} ,\qquad
0 <\beta < \frac{2}{L_f} ,\qquad
\alpha \beta + \frac{\alpha L_g}{2} < 1,
\end{align*}
the iterative process defined by \eqref{AlgI} converges to a saddle point of $L(u,p)$ in the set $C\times K$.
\end{theorem}
\bigskip
\begin{proof}
We maintain using the notation \eqref{nota1} and introduce some more notation:
\begin{align*}
\widetilde q_n = A \widetilde u_n, \quad q_n = A u_n,\quad Q_n = A U_n, \quad \ov Q_n = A\ov U_n.
\end{align*}
Under the above settings and passing the translation of projection \eqref{trPi}, we can rewrite the process \eqref{AlgI} as
\begin{align*}
\begin{cases}
P_{n+1} = \Pi_{\widetilde K} (P_n +\alpha (\ov q_n - G'(p_n)))\\
U_{n+1} = \Pi_{\widetilde C} (U_n - \beta \widetilde u_n)\\
q_{n+1} = A u_{n+1}\\
\ov q_{n+1} = 2q_{n+1} -q_n.
\end{cases}
\end{align*}
We manipulate again the property of projection  \eqref{Pythagorean}, the characterization \eqref{chaP} and the representation \eqref{dg2} to obtain that
\begin{align*}
\lVert P_{n+1} \rVert ^2 
&\le \lVert P_n +\alpha (\ov q_n - G'(p_n)) \rVert^2 - \lVert P_{n+1} -  P_n -\alpha (\ov q_n - G'(p_n) ) \rVert^2 \\
&\le \lVert P_n \rVert^2 - \lVert P_{n+1} - P_n \rVert^2 + 2\alpha \langle  \ov Q_n - g_n P_n, P_{n+1} \rangle
\end{align*}
and similarly,
\begin{align*}
\lVert U_{n+1} \rVert ^2 
&\le \lVert U_n - \beta \widetilde u_n \rVert^2 - \lVert U_{n+1} -  U_n + \beta \widetilde u_n \rVert^2 \\
&\le \lVert U_n \rVert^2 - \lVert U_{n+1} - U_n \rVert^2 - 2\beta \langle \widetilde u_n, U_{n+1} \rangle.
\end{align*}
Within the spirit of Lemma \eqref{ImpProj}, $Q_{n+1}$ is indeed the projection of $(Q_n -\beta \widetilde q_n)$ on $A \widetilde C$. We then have
\begin{align*}
\lVert Q_{n+1} \rVert ^2 
&\le \lVert Q_n - \beta \widetilde q_n \rVert^2 - \lVert Q_{n+1} -  Q_n + \beta \widetilde q_n \rVert^2 \\
&\le  \lVert Q_n \rVert^2 - \lVert Q_{n+1} - Q_n \rVert^2 - 2\beta \langle \widetilde q_n, Q_{n+1} \rangle
\end{align*}
We denote that
\begin{align*}
S_n &= \alpha (\lVert U_n\rVert^2  +\lVert Q_n\rVert^2) + \beta \lVert P_n \rVert^2 \\
R_n &= \alpha (\lVert U_{n+1} - U_n \rVert^2 + \lVert Q_{n+1} - Q_n \rVert^2) + \beta \lVert P_{n+1} - P_n \rVert^2
\end{align*}
Then, it holds
\begin{align}
S_{n+1} \le S_n -R_n -2\alpha\beta \left( \langle  \widetilde u_n, (A^*A + I)U_{n+1} \rangle - \langle  \ov Q_n - g_n P_n, P_{n+1} \rangle \right). \label{Sn}
\end{align}
Keeping in mind that $\widetilde u_n = (A^*A +I)^{-1}(A^*p_{n+1} + F'(u_n))$, $(A^*A +I)$ is self-adjoint, and use the inequality \eqref{chaU} and the representation \eqref{df2}, we get
\begin{align*}
\langle \widetilde u, (A^*A + I)U_{n+1} \rangle 
= & \langle (A^*A +I)^{-1}(A^*p_{n+1} + F'(u_n)), (A^*A + I)U_{n+1} \rangle \\
= & \langle A^*p_{n+1} + F'(u_n), U_{n+1} \rangle \\
\ge & \langle A^* P_{n+1} + f_n U_n, U_{n+1} \rangle.
\end{align*}
Besides, we derive that
\begin{align*}
& \langle A^* P_{n+1} + f_n U_n, U_{n+1} \rangle - \langle  \ov Q_n - g_n P_n, P_{n+1} \rangle \\
= & \langle  f_n U_n, U_{n+1} \rangle +\langle g_nP_n, P_{n+1}\rangle - \langle \ov Q_n -  Q_{n+1}, P_{n+1}\rangle \\
=& \langle M_{f_n} U_n, M_{f_n} U_{n+1} \rangle +\langle M_{g_n} P_n, M_{g_n} P_{n+1}\rangle \\
& + \langle Q_{n+1} - Q_n, P_{n+1}\rangle - \langle Q_n -  Q_{n-1}, P_n\rangle 
- \langle Q_n -  Q_{n-1}, P_{n+1} - P_n\rangle \\
\ge & -\frac{L_f}{4}\lVert U_{n+1} - U_n \rVert^2 - \frac{L_g}{4}\lVert P_{n+1} - P_n \rVert^2 \\
& + \langle Q_{n+1} - Q_n, P_{n+1}\rangle - \langle Q_n -  Q_{n-1}, P_n\rangle 
- \lVert Q_n -  Q_{n-1} \rVert \lVert P_{n+1} - P_n\rVert
\end{align*}
In the inequality above, we reused the estimates \eqref{estimF}-\eqref{estimG} in the proof of congvergence of explicit scheme.
We recall that $2ab \le (\delta a^2 + b^2/\delta)$ for any $a,b$ and any $\delta >0$. Therefore, for every $\delta >0$, it follows
\begin{align*}
S_{n+1}
\le  S_n & - \alpha\lVert Q_{n+1} - Q_n \rVert^2 -\alpha  \left(1- \frac{\beta L_f}{2} \right) \lVert U_{n+1} - U_n \rVert^2 \\
& - \beta  \left(1- \frac{\alpha L_g}{2}  \right)\lVert P_{n+1} - P_n \rVert^2 \\
&-2\alpha\beta \left[ \langle Q_{n+1} - Q_n, P_{n+1}\rangle - \langle Q_n -  Q_{n-1}, P_n\rangle \right] \\
&+ 2\alpha\beta  \Vert Q_n - Q_{n-1} \rVert \lVert P_{n+1} - P_n \rVert,
\end{align*}
and then
\begin{align*}
S_{n+1} \le  S_n & - \alpha\lVert Q_{n+1} - Q_n \rVert^2 -\alpha  \left(1- \frac{\beta L_f}{2} \right) \lVert U_{n+1} - U_n \rVert^2  \\
& - \beta \left(1- \frac{\alpha L_g}{2} \right)\lVert P_{n+1} - P_n \rVert^2\\
&-2\alpha\beta \left[ \langle Q_{n+1} - Q_n, P_{n+1}\rangle - \langle Q_n -  Q_{n-1}, P_n\rangle \right] \\
&+ \alpha\beta \left( \delta\Vert Q_n - Q_{n-1} \rVert^2 + \frac{1}{\delta}\lVert P_{n+1} - P_n \rVert^2 \right).
\end{align*}
Let $M\ge N \ge 1$ and $\gamma >0$. We take the sum of the inequality \eqref{Sn} from $N$ up to $M$ in order to obtain
\begin{align*}
S_{M+1} \le S_{N} & -\alpha\sum_{n=N}^{M} \lVert Q_{n+1} - Q_n \rVert^2 \\
& -\alpha \left(1- \frac{\beta L_f}{2} \right)\sum_{n=N}^{M} \lVert U_{n+1} - U_n \rVert^2 - \beta \left(1- \frac{\alpha L_g}{2}  \right) \sum_{n=N}^{M} \lVert P_{n+1} - P_n \rVert^2 \\
& -2\alpha\beta \left[ \langle Q_{M+1} - Q_M, P_{M+1}\rangle - \langle Q_N -  Q_{N-1}, P_N\rangle \right] \\
&+ \alpha\beta \delta \sum_{n=N-1}^{M-1} \Vert Q_{n+1} - Q_n \rVert^2 + \frac{\alpha\beta}{\delta} \sum_{n=N}^{M} \lVert P_{n+1} - P_n \rVert^2,
\end{align*}
which leads to
\begin{align*}
S_{M+1}
\le S_{N} & -\alpha (1 - \beta\delta)\sum_{n=N}^{M-1} \lVert Q_{n+1} - Q_n \rVert^2 
-\alpha \left(1- \frac{\beta L_f}{2} \right)\sum_{n=N}^{M} \lVert U_{n+1} - U_n \rVert^2\\
& - \beta \left(1- \frac{\alpha L_g}{2} -\frac{\alpha}{\delta} \right) \sum_{n=N}^{M} \lVert P_{n+1} - P_n \rVert^2 \\
&-\alpha\lVert Q_{M+1} - Q_M \rVert^2 +  \alpha\beta \delta \Vert Q_N - Q_{N-1} \rVert^2 \\
& +2\alpha\beta \langle Q_N -  Q_{N-1}, P_N\rangle \\ 
& +\alpha\beta \left( \gamma\Vert Q_{M+1} -Q_{M} \rVert^2 + \frac{1}{\gamma}\lVert P_{M+1} \rVert^2 \right).
\end{align*}
After rearranging, it turns out immediately that
\begin{align}
\begin{aligned}
\alpha \lVert Q_{M+1} \rVert^2 
&+ \alpha \lVert U_{M+1} \rVert^2 + \beta \left(1- \frac{\alpha}{\gamma} \right) \lVert P_{M+1} \rVert^2 \\ 
&+\alpha(1-\beta \delta)\sum_{n=N}^{M-1} \lVert Q_{n+1} - Q_n \rVert^2\\
&+\alpha \left(1- \frac{\beta L_f}{2} \right)\sum_{n=N}^{M} \lVert U_{n+1} - U_n \rVert^2\\
&+ \beta \left(1- \frac{\alpha L_g}{2} -\frac{\alpha }{\delta} \right) \sum_{n=N}^{M} \lVert P_{n+1} - P_n \rVert^2 \\
&+\alpha(1-\beta \gamma)\lVert Q_{M+1} - Q_M \rVert^2 \\
\le S_{N} &
+  \alpha\beta  \delta \Vert Q_N -Q_{N-1} \rVert^2 +2\alpha\beta \langle Q_N -  Q_{N-1}, P_N\rangle.
\end{aligned}
\label{mainConvI}
\end{align}
We extract all the coefficients in the left-hand side of the inequality \eqref{mainConvI} and regard that for any $\alpha, \beta, \delta, \gamma$ being positive parameters, the following holds true
\begin{align}
\begin{cases}
\beta \left(1- \frac{\alpha}{\gamma} \right) > 0 \\
\alpha(1-\beta \delta) > 0 \\
\alpha \left(1- \frac{\beta L_f}{2} \right) > 0 \\
\beta \left(1- \frac{\alpha L_g}{2} -\frac{\alpha}{\delta} \right) > 0 \\
\alpha(1-\beta \gamma) > 0
\end{cases}
\quad
\Longleftrightarrow
\quad
\begin{cases}
\beta < \frac{2}{L_f} \\
\alpha < \gamma < \frac{1}{\beta}\\
\beta < \frac{1}{\delta } <  \frac{2- \alpha L_g}{2\alpha }.
\end{cases}\label{IntConvI}
\end{align}
We then derive from the right-hand side of \eqref{IntConvI} that
\begin{align}
0<\alpha <\frac{2}{L_g}, \qquad
0 <\beta < \frac{2}{L_f}, \qquad
\alpha \beta + \frac{\alpha L_g}{2} < 1.
\label{paramConvI}
\end{align}
So, if we choose the couple $(\alpha,\beta)$ as in \eqref{paramConvI} and $(\gamma,\delta)$ in the intervals defined by the right-hand side of \eqref{IntConvI} then the left-hand side of \eqref{mainConvI} is positive. We deduce that the sequences $\lVert U_{M+1} \rVert$, $\lVert Q_{M+1} \rVert$, $\lVert P_{M+1} \rVert$ are bounded while $\lVert U_{M+1} - U_M \rVert$, $\lVert Q_{M+1} - Q_M \rVert$ and $\lVert P_{M+1} - P_M \rVert$ must converge to 0 as $M\to +\infty$. We then have the same conclusion for the sequences $\lVert u_{M+1} \rVert$, $\lVert q_{M+1} \rVert$, $\lVert p_{M+1} \rVert$, $\lVert u_{M+1} - u_M \rVert$, $\lVert q_{M+1} - q_M \rVert$ and $\lVert p_{M+1} - p_M \rVert$, respectively. Therefore, the sequences $\{ u_{M+1}\} $, $\{ q_{M+1} \}$, $\{ p_{M+1} \}$ have subsequences converging in weak topology. Let say $u_*$, $q_*$, and $p_*$ are corresponding limits. By substituting $\hat u= u_*$, $\hat q = q_*$, $\hat p = p_*$, and handling again the inequality \eqref{mainConvI} we can derive that sequences $\{ u_M \}$, $\{q_M\}$, $\{ p_M \}$ are indeed Cauchy sequences, hence converge to the limits $u_*$, $q_*$, and $p_*$, respectively. By passing to the limit in \eqref{AlgI} and  \eqref{ProjP}, we have
\begin{align*}
p_* & = \Pi_K ( p_* +\alpha (Au_*- G'(p_*))\\
u_* & = \Pi_C (u_* - \beta \widetilde u_* )\\
A u_* & = \Pi_{A C}( A u_* - \beta A \widetilde u_*)
\end{align*}
where $\widetilde u_*$ is the proximal point $\prox_{A,p_*}(F'(u_*))$ (see \eqref{utilde}), that is
\begin{align*}
 \widetilde u_* = (A^*A + I)^{-1}(A^* p_* + F'(u_*)).
\end{align*}
It is easy to see that for every $u\in C$
\begin{align*}
0 \le \langle \widetilde u_*, u - u_* \rangle + \langle A \widetilde u_* , A (u-u_*) \rangle = \langle (A^* A +I) \widetilde u_*, u - u_* \rangle.
\end{align*}
We deduce that
\begin{align*}
\langle A u_* - G'(p_*) , p - p_* \rangle & \le 0 \quad \forall p\in K, \\
\langle A^* p_* + F'(u_*),u - u_* \rangle & \ge 0 \quad \forall u\in  C.
\end{align*}
This is to say that $(u_*,p_*)$ is a saddle point  of $L(u,p)$ in $C\times K$.
\end{proof}

\begin{remark}
We emphasize that if the term $F(u)$ is absent in the Lagrangian $L(u,p)$, the projection on $\ddot C$ can be restricted only on $AC$. Then, the algorithm \eqref{AlgI} reads
\begin{align*}
\begin{cases}
p_{n+1} = \Pi_K ( p_n +\alpha (\ov q_n - G'(p_n))\\
q_{n+1} = \Pi_{AC} (q_n - \beta A(A^*A)^{-1}A^*p_{n+1} )\\
\ov q_{n+1} = 2q_{n+1} -q_n
\end{cases}
\end{align*}
where $q_n = A u_n$. In this situation, the hypothesis \eqref{Ha} should be replaced by $A^*A$ being positive definite so that $A^*A$ is invertible. And when the term $G(p)$ is not also present, the positive parameters $\alpha, \beta$ just have to satisfy the constraint $\alpha\beta <1$ to ensure the convergence of the algorithm.

\end{remark}

\bigskip
\section{Application to the shape optimization of thin torsion rods}
\label{appli}
Let $D$ be a bounded connected domain in $\RR^2$ and $s$ be a real parameter. We are interested in considering the variational problem, studied in \cite{alibert2013nonstandard}
\begin{align*}
m(s):= \inf \left\{ \int_D \varphi(\nabla u) : u\in H^1_0(D), \int_D u = s \right\}
\end{align*}
where $\varphi:\RR^2 \to \RR$ is a convex function given by
\begin{align*}
\varphi(z) := 
\begin{cases}
\frac{1}{2}(|z|^2 +1) &\tinm{if} |z| \ge 1 \\
|z| &\tinm{if} |z| \le 1.
\end{cases}
\end{align*}
We see that the integrand $\varphi$ is not strictly convex and not differentiable at $0$. Its Fenchel conjugate is the positive part of a quadratic form
\begin{align*}
\varphi^*(p)=\frac{1}{2} \left( |p|^2-1 \right)_+.
\end{align*}
It is clear that $\varphi^*$ is convex but non strictly convex, too. See Figure \ref{fig:varphi} for illustration.

\begin{figure}[ht]
\begin{center}
\includegraphics[scale=0.83]{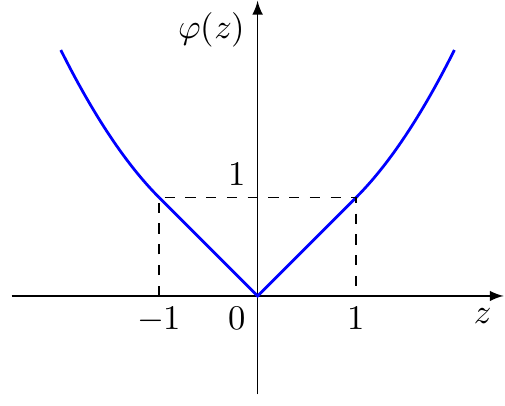}
\includegraphics[scale=0.83]{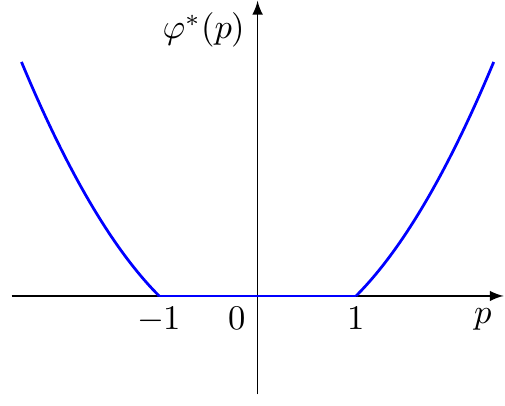}
\includegraphics[scale=0.83]{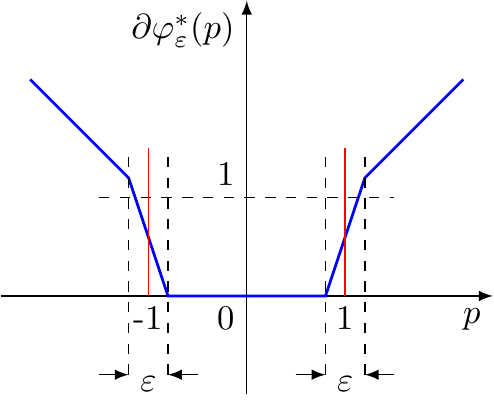}
\end{center}
\caption{ From left to right, the description of function $\varphi(z)$, its Fenchel conjugate $\varphi^*(p)$ and the regularization of $\varphi^*(p)$.}
\label{fig:varphi}
\end{figure}

We recall that the Fenchel conjugate of an integral functional is an integral of the Fenchel conjugate of the corresponded integrand, i.e. 
\begin{align*}
 \left( \int_D f(p(x)) dx \right)^* = \int_D f^*(p^*(x)) dx.
\end{align*}
For more details on this topic, we refer to \cite{ekeland1974analyse}. We are going to apply this fact to calculate the conjugate of $m(s)$. For every $\lambda\in\RR$,
\begin{align*}
m^*(\lambda) & = \sup_{s\in\RR} \left\{ \lambda s - m(s) \right\} = -\inf_{u\in H^1_0(D)} \left\{ \int_D \varphi(\nabla u) -\lambda \int_D u \right\} \\
& = -\inf_{u\in H^1_0(D)} \sup_{p\in L^2(D;\RR^2)} \left\{ \int_D p\cdot \nabla u - \int_D \varphi^*(p) -\lambda \int_D u \right\} 
\end{align*}
The second equality occurs with a replacement of the functional $\int_D \varphi(\nabla u)$ by a Fenchel conjugate. We regard that  the functional 
\begin{align*}
p\mapsto \inf_{u\in H^1_0(D)} \left\{ \int_D p\cdot \nabla u - \int_D \varphi^*(p) -\lambda \int_D u \right\}
\end{align*}
is finite if and only if $-\Div p = \lambda$. So, by passing the inf-sup permutation argument, $m^*(\lambda)$ can be rewritten as
\begin{align}
m^*(\lambda) = \inf \left\{ \int_D \varphi^*(p) : p\in L^2(D;\RR^2), -\Div p = \lambda \right\}.
\end{align}
Furthermore, optimal solutions of $m(s)$ and $m^*(\lambda)$ are characterized by certain optimality conditions
\begin{align*}
\begin{cases}
u \tinm{solution to} m(s) \\
p \tinm{solution to} m^*(\lambda) \\
\lambda\in \partial m(s)
\end{cases}
\qquad \Longleftrightarrow
\qquad
\begin{cases}
\int_D u =s \\
-\Div p = \lambda \\
p \in \partial \varphi(\nabla u) \tinm{a.e.}
\end{cases}
\end{align*}
Many interesting properties of functions $m(s)$ and $m^*(\lambda)$ were studied in \cite{alibert2013nonstandard}. One of them which is obviously seen is that the Fenchel equality is satisfied 
\begin{align*}
m(s) + m^*(\lambda) = s \lambda,
\end{align*}
since $\lambda \in \partial m(s)$.

We shall focus on the inf-sup formulation of $m^*(\lambda)$ which is adapted within our numerical schemes. But we must notice that $\varphi^*$ is neither strictly convex and nor differentiable on the unit circle $\{p\in\RR^2: |p|=1 \}$. A regularization for $\varphi^*$ should be done before enforcing the algorithms. See Figure \ref{fig:varphi} for visualization. Instead of taking the subgradient of $\varphi^*$, we regularize it by removing the discontinuities with an $\e$-affine symmetric connection
\begin{align*}
\partial \varphi^*_\e(p) 
=
\begin{cases}
\frac{2 + \e }{2\e} \left( |p| + \frac{\e}{2} -1 \right) \frac{p}{|p|}  &\tinm{if} ||p|-1| \le \frac{\e}{2}\\
\partial \varphi^*(p) &\tinm{otherwise.}
\end{cases}
\end{align*}
It is ready to find a saddle point of the problem with $\e$-regularization.

The solution of such a problem in the context of shape
optimization of thin torsion rods exhibits regions where $u$ is constant corresponding to regions without material, regions where $|\nabla u|\ge 1$ corresponds to the optimal region for the material in order to struggle torsion. Regions where $0<|\nabla u|< 1$ describes the regions of homogenized material for which the convexity of the Lagrangian is not strict. This makes the problem nontrivial. Depending on the mass constraint, such a homogenization region can appear (low mass constraint leading to the so-called "homogeneous solution") or not ("special solution").
To answer the question whether an optimal design contains some homogenization region is equivalent to investigate when the special solution exists. And naturally, we wonder in which domain $D$ special solutions present. These are still open issue.
In Figure \ref{figsol_special_homog}, the magnitude of the gradient of the solution is plotted, a special solution can be found in the left figure, a homogenized solution is in the middle with weak gradient (lower than $1$) on regions limited by thick lines corresponding to $|\nabla u|=1$. In the right figure, the value $\lambda$ is close to the critical Cheeger constant while as we see, the optimal shape becomes thinner and tends to the boundary of the Cheeger set of the domain $D$ \cite{kawohl2006characterization}.
\begin{figure}[ht]
\begin{center}
\includegraphics[scale=0.115]{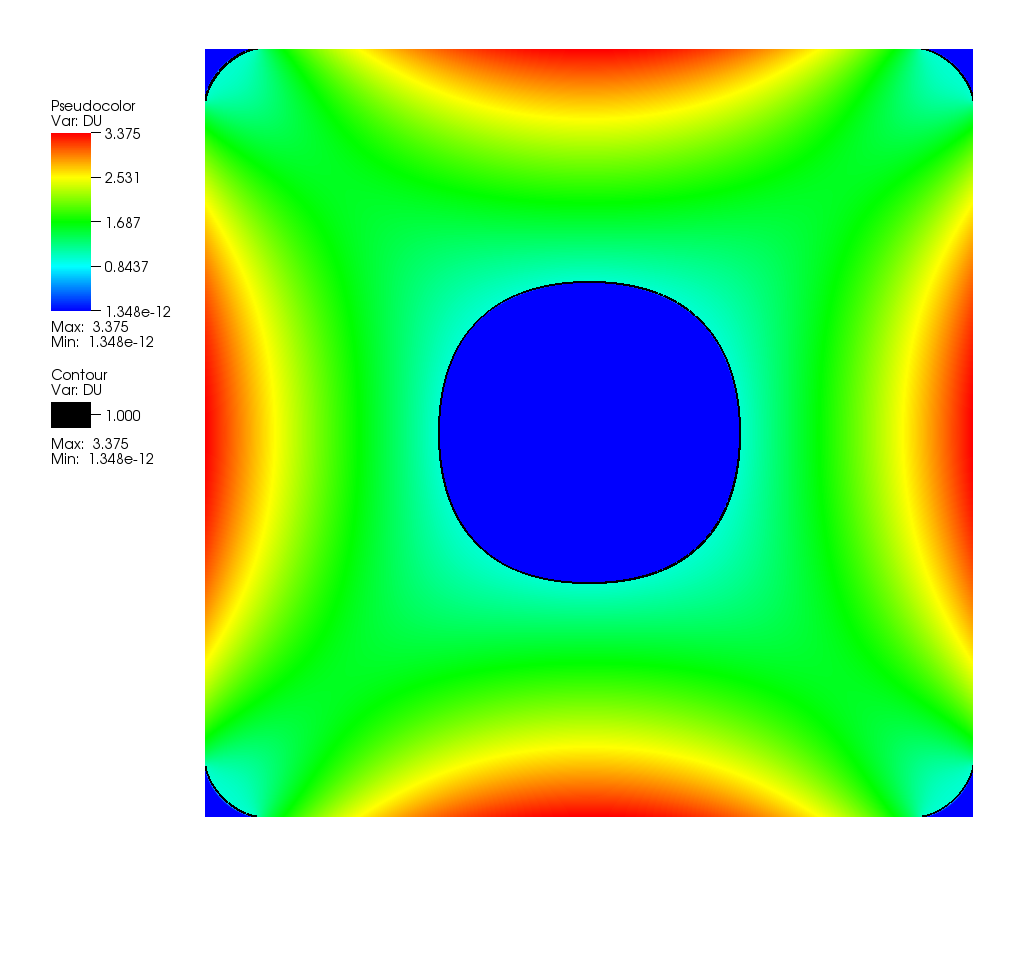}
\includegraphics[scale=0.115]{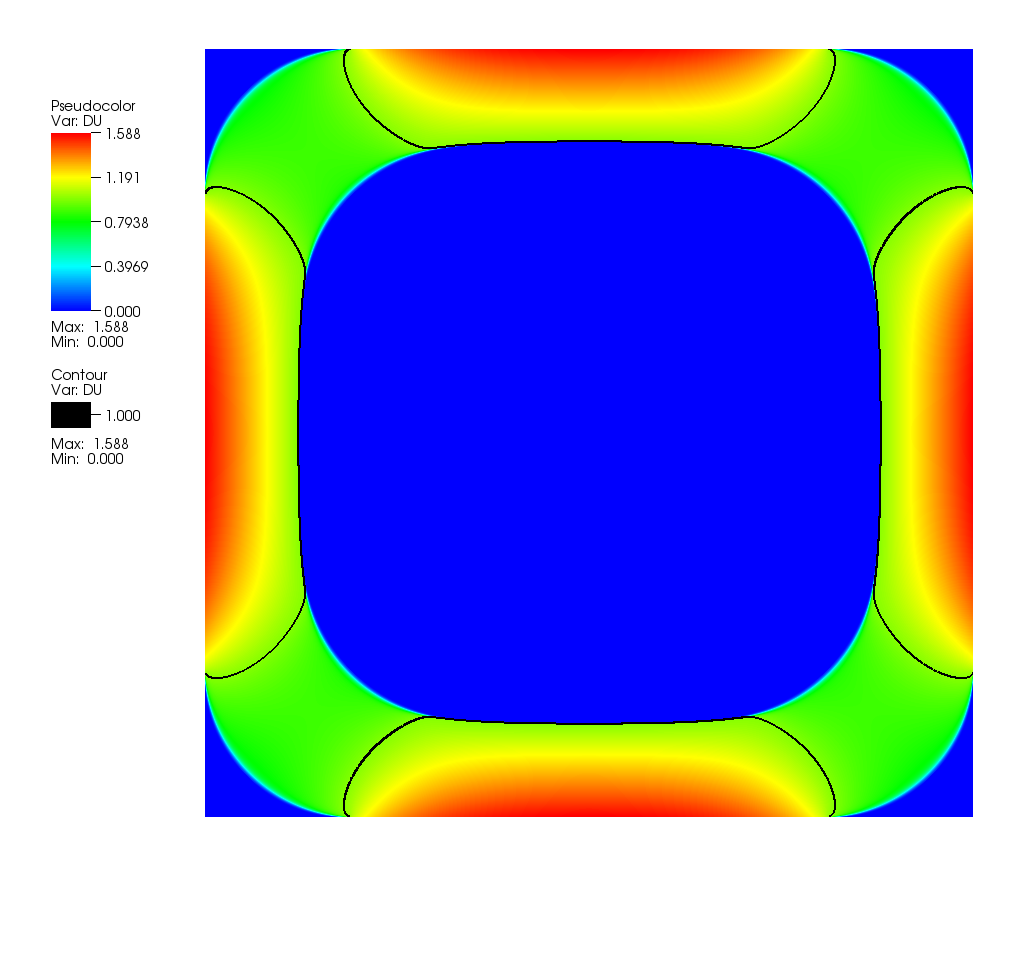}
\includegraphics[scale=0.115]{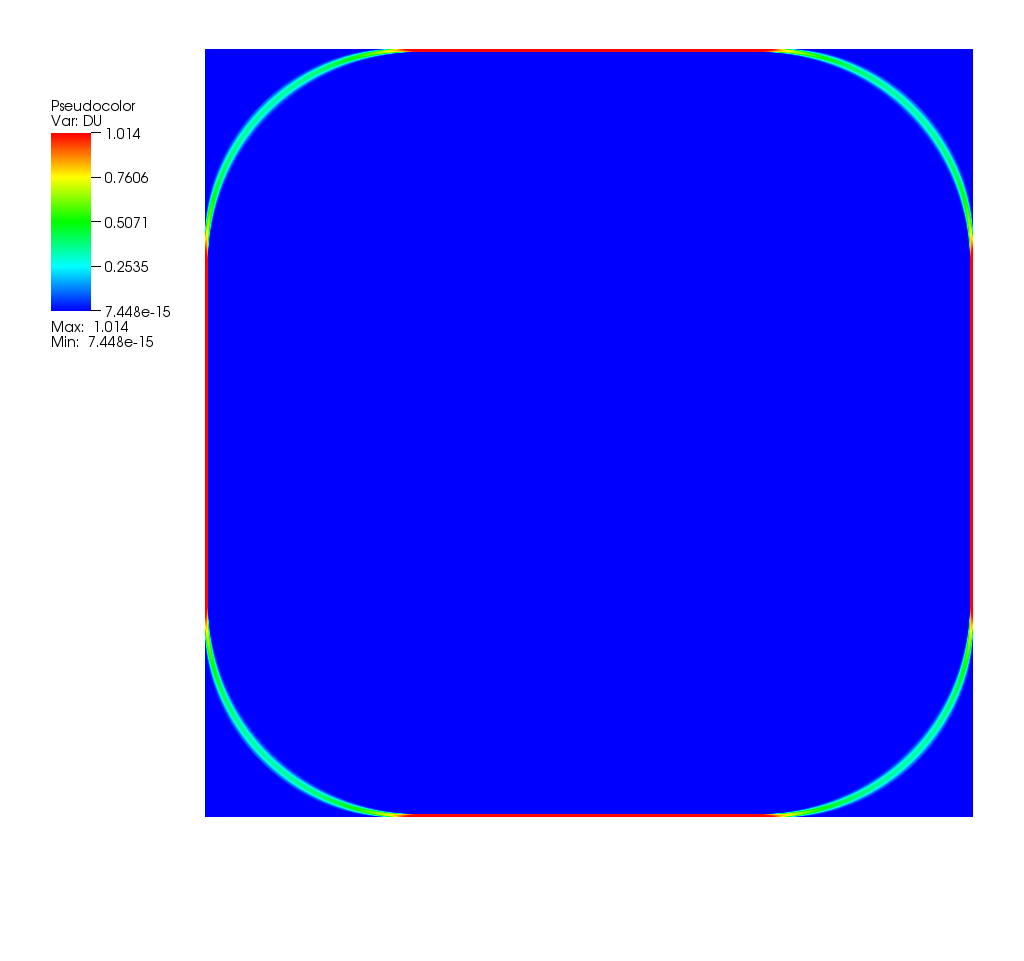}
\end{center}
\caption{\label{figsol_special_homog}Special (left) and homogenized (middle and right) solutions depending on the mass constraint.}
\end{figure}

It is easy to see that if $D$ is a symmetric simple connected domain then solutions of $m(s)$ are symmetric. When $D$ is a square, the symmetrization allows dividing $D$ in four parts. Without lost of generality, let us describe the discretization settings with $D=[0,0.5]^2$ instead of unit square. Let $x=(x_1,x_2)$, $p=(p_1,p_2)$. The subdivision leads to the appearance of extra boundary conditions
\begin{align*}
u(\cdot,0.5) = u(0.5, \cdot) = 0 \\
p_1(0,\cdot) = p_2(\cdot,0) = 0 .
\end{align*}

We shall implement our algorithms with space discretization on staggered MAC  grids \cite{harlow1965numerical}. This choice leads to discretizations satisfying $\Div^h = -(\nabla^h)^*$, where the superscript denote the discretization with a mesh size $h$. $\nabla^h$ stands for discrete gradient operator.

We provide an explicit iterative process in discrete scheme with
\begin{align}
\label{ES}
\begin{cases}
p_{n+1}^h = p_n^h + \alpha (\nabla^h \overline{u}_n^h - \partial \varphi^*_\e (p_n^h)) \\
u_{n+1}^h = \Pi_C^h( u_n^h + \beta (\Div^h (p^h_{n+1}) + \lambda))\\
\overline{u}^h_{n+1} = 2 u^h_{n+1} - u^h_n.
\end{cases}
\end{align}
where $\Pi_C^h$ is the discretized projection. In this case, the projection is just simple to keep the boundary condition $ u = 0$ on $\partial D$. Since Lipschitz constants $L_f=0$, $ L_g= (2 + \e)/ 2\e$, the positive parameters $\alpha,\beta$ should be chosen such that 
\begin{align*}
\alpha < \frac{4\e}{2+\e} , \qquad \alpha\beta c_h^2 + \frac{\alpha(2+\e)}{4\e} \le 1 \qquad \tinm{with} c_h :=||\nabla^h||= \frac{2\sqrt{2}}{h}.
\end{align*}
We remark that $\alpha < 4$ and tends to $0$ as $\e \to 0$. But with the second condition, the product $\alpha\beta$ should be of order $O(h^2)$.

In the point of view of implicit scheme, the proposed convergent process reads
\begin{align}
\begin{cases}
p_{n+1}^h = p_n^h + \alpha (\nabla^h \overline{u}_n^h - \partial \varphi^*_\e (p_n^h)) \\
u_{n+1}^h =  u_n^h + \beta (I-\Delta^h)^{-1}(\Div^h (p^h_{n+1}) + \lambda)\\
\overline{u}^h_{n+1} = 2 u^h_{n+1} - u^h_n.
\end{cases}
\label{IS}
\end{align}
We replaced $\Div^h (p^h_{n+1}) + \lambda$ in the explicit process by $\widetilde u_n = (I-\Delta^h)^{-1}(\Div^h (p^h_{n+1}) + \lambda)$ which is solution of the equations
\begin{align*}
\begin{cases}
(I-\Delta^h) v = \Div^h (p^h_{n+1}) + \lambda \\
v = 0 \quad \tinm{on} \quad \partial D
\end{cases}
\end{align*}
The projection $\Pi_C^h$ then disappears since the boundary condition on $u$ is added within resolving $\widetilde u_n$. Moreover, the positive parameters are simplified 
\begin{align*}
\alpha < \frac{4\e}{2+\e} , \qquad \alpha\beta + \frac{\alpha(2+\e)}{4\e} < 1.
\end{align*}
We see that the choices of $\alpha,\beta$ now do not depend on $c_h$ and thus, the product $\alpha\beta$ is of order $O(1)$ with respect to $h$. Nevertheless, the step size $\alpha$ is still restricted by small $\e$. 

The regularizing parameter $\e$ is linked to the grid size, in practice we take for instance $\e=3h$. This relation makes the step size $\alpha$ be of order of $h$  for the implicit algorithm (\ref{IS}). It leads to introduce a new algorithm with $\kappa$ sub-iterations on the explicit part of  (\ref{IS}), in implementation $\kappa=50$:
\begin{align}
\begin{cases}
p_{n+1,0}^h = p_{n,\kappa}^h\\ 
p_{n+1,k+1}^h = p_{n+1,k}^h + \frac{\alpha}{\kappa} (\nabla^h \overline{u}_n^h - \partial \varphi^*_\e (p_{n+1,k}^h)), \qquad k=0,1,...,\kappa-1 \\
u_{n+1}^h =  u_n^h + \beta (I-\Delta^h)^{-1}(\Div^h (p_{n+1,\kappa}^h) + \lambda)\\
\overline{u}^h_{n+1} = 2 u^h_{n+1} - u^h_n.
\end{cases}
\label{ISS}
\end{align}
The algorithm (\ref{ISS}) allows $\alpha$ to be $\kappa$  times bigger and then reduces the number of iterations in $n$.

From now on, the algorithm (\ref{ISS}) is called Implicit Sub-iteration Scheme (ISS), the algorithm (\ref{IS}) is called Implicit Scheme (IS) and the algorithm (\ref{ES}) is called Explicit Scheme (ES).

\begin{figure}[ht]
\begin{center}
\includegraphics[scale=0.78]{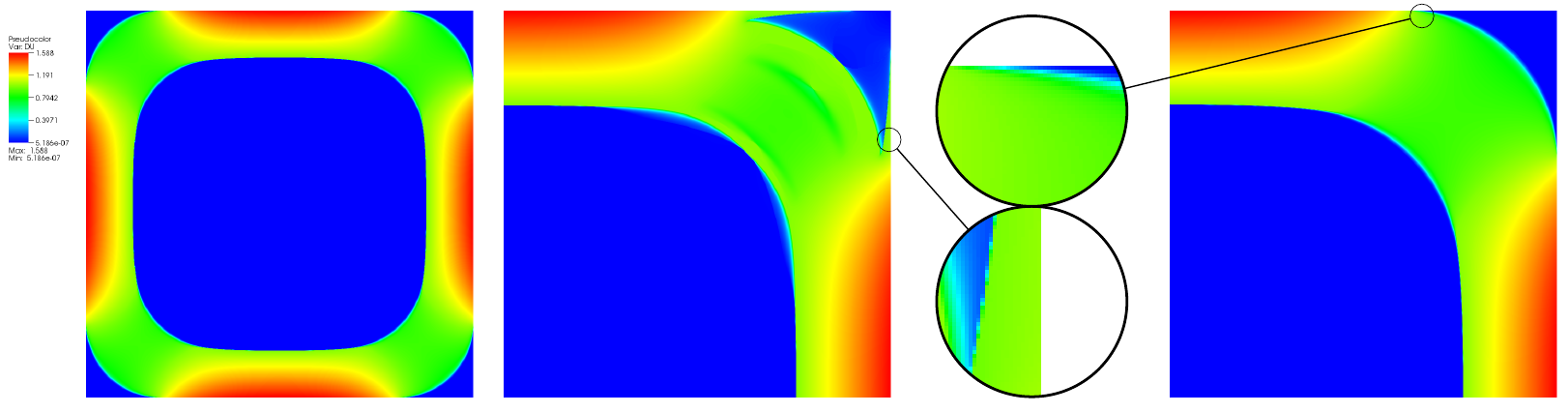}
\end{center}
\caption{Case $\lambda=5$ for (ES) on the left, (IS) in the middle, (ISS) on the right, for the same grid size.}\label{lambda5}
\end{figure}

In Figure \ref{lambda5}, we present  the Euclidean norm of the gradient of a solution, i.e. $|\nabla u|$. The left figure shows an expected solution with the explicit scheme. The centered and the right ones are in a quarter domain which are done with implicit methods (IS) and (ISS), respectively. We see that in  implicit scheme (IS), the contact zones are still not tangent to the boundary of domain (the region on which it is difficult to converge the algorithm). And, the (IS) algorithm, expected to be slower than (ISS), introduces an other drawback with numerical artefacts. 

In the processes \eqref{IS}-\eqref{ISS}, the inverse Laplacian computation is the most costly. The computational cost thus depends highly on the solver used for the inverse Laplacian operator. If one uses a multigrid or a FFT solver, it can be of order $\displaystyle \frac{1}{h^{2} \log h}$. Besides, handling a multigrid solver, for instance AGMGPAR (A parallel version of Algebraic Multigrid method), see \cite{notay2010user}, and MAC scheme, it easy to implement our algorithms with MPI (Massage Passing Interface) library which provides an effective environment for parallel computation. 
In the following computations,  comparisons between the different schemes are done with a  fixed number of process to $6$, leading to good scalability for all methods.

Before comparing the computational cost, we ensure that the algorithms converge to the exact solution as the grid size $h$ goes to zero. We then consider a case in which the unique exact solution is known, that is when $D$ is a disk.
In such a case, the solution is a special solution and it is radial  with  mass concentrated on the periphery and with an internal radius of $\overline R=\frac 2 \lambda$ (see Figure \ref{exact}), see \cite{alibert2013nonstandard} for the expression of the exact solution.
 \begin{figure}[ht]
\begin{center}
\includegraphics[trim={1.2cm 4.cm 1.cm 1.cm},clip,scale=0.115]{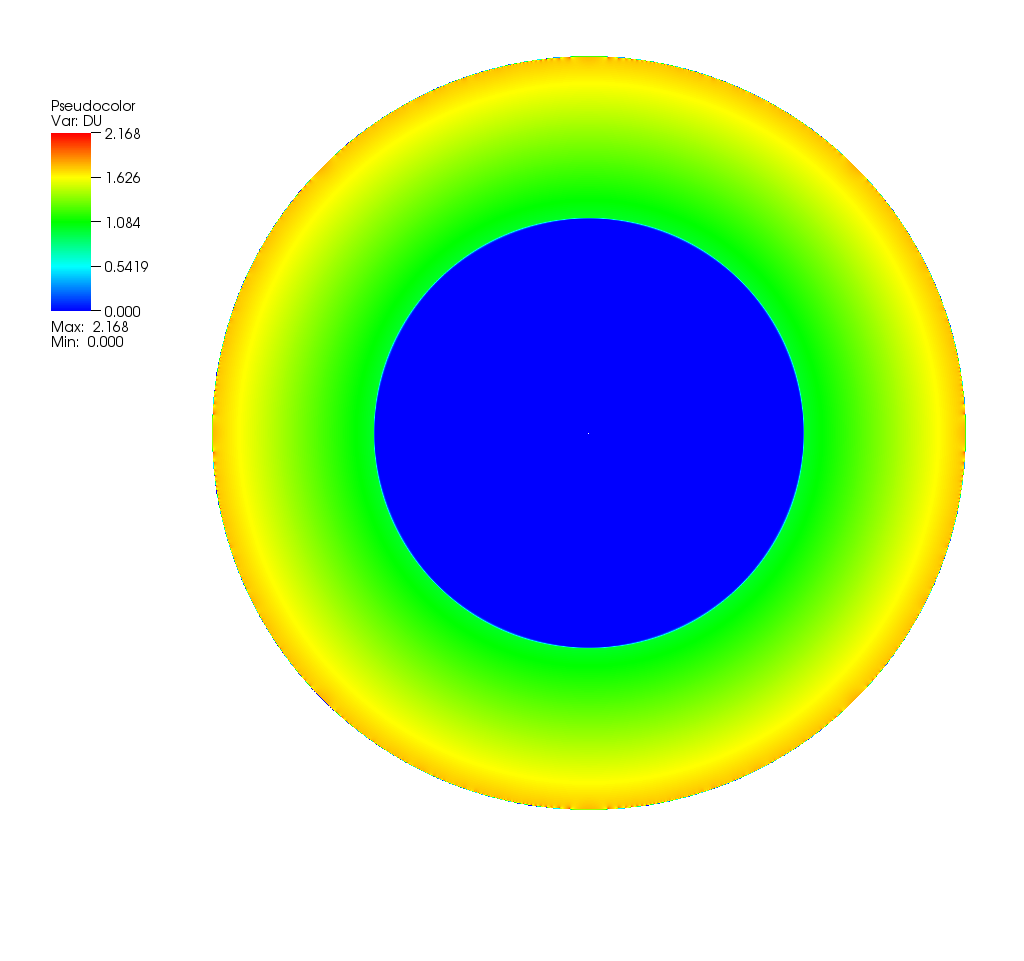}
\end{center}
\caption{Exact solution when the constraint domain $D$ is a disk.}\label{exact}
\end{figure}
We are then able to compute the numerical internal radius $r_N$, $R_N$ defined respectively as the maximal value of radius where $|\nabla u|$ is smaller than a half and the minimal value of radius where $|\nabla u|$ is bigger than one. The error on the internal radius is measured as a number of cells for different cell sizes. The radius error is of order of a half of grid size whatever the grid, as shown on Figure \ref{radius}.

\begin{figure}[ht]
\begin{center}
\includegraphics[scale=0.7]{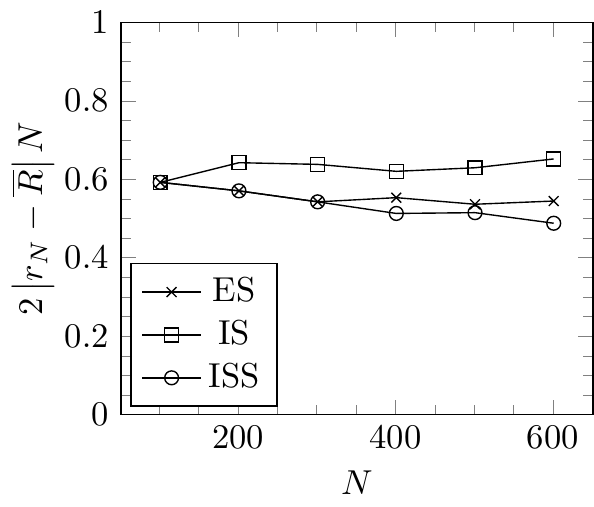}\qquad 
\includegraphics[scale=0.7]{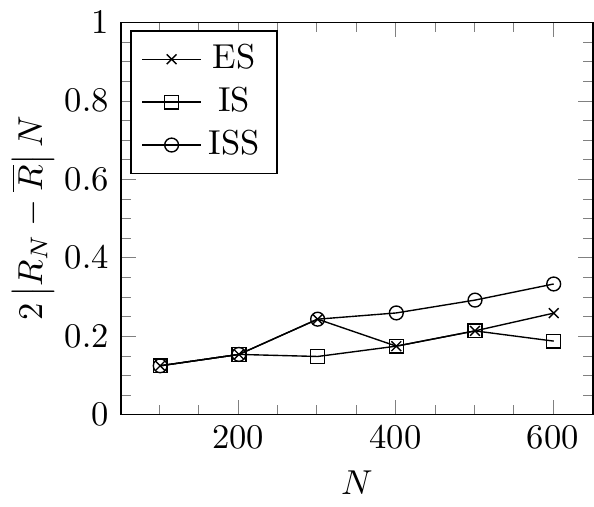}
\end{center}
\caption{Graph of $2 {|r_N-\overline R|}{N}$ and $ 2{|R_N-\overline R|}{N}$ with respect to the grid size $N$.\label{radius}}
\end{figure}

We are now concerned with the comparison of computational cost between schemes (ES), (IS), (ISS) in the case where $\lambda=5$,  for a unitary square $D$, so that an homogeneous solution occurs.

\begin{figure}[ht]
\begin{center} 
\includegraphics[trim={1.cm 0.cm 1.cm 1.cm},clip,scale=0.115]{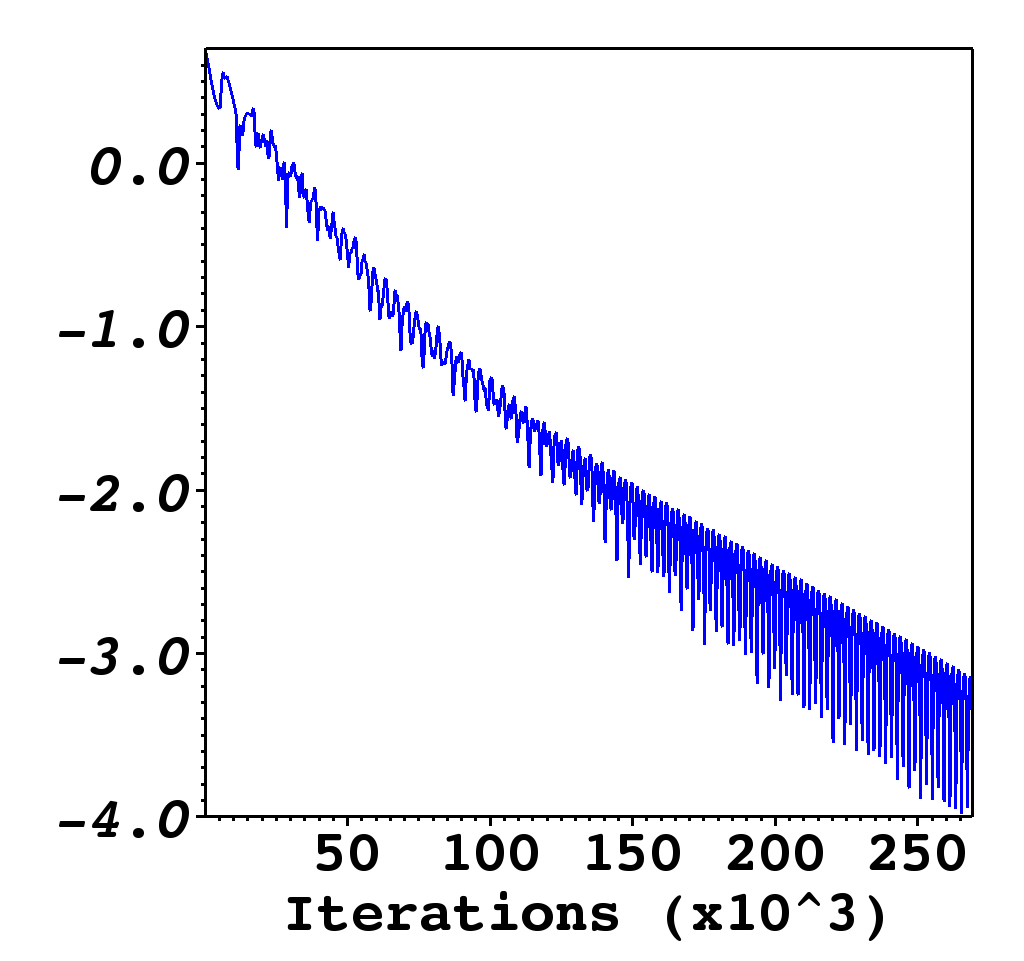} \quad
\includegraphics[trim={1.cm 0.cm 1.cm 1.cm},clip,scale=0.115]{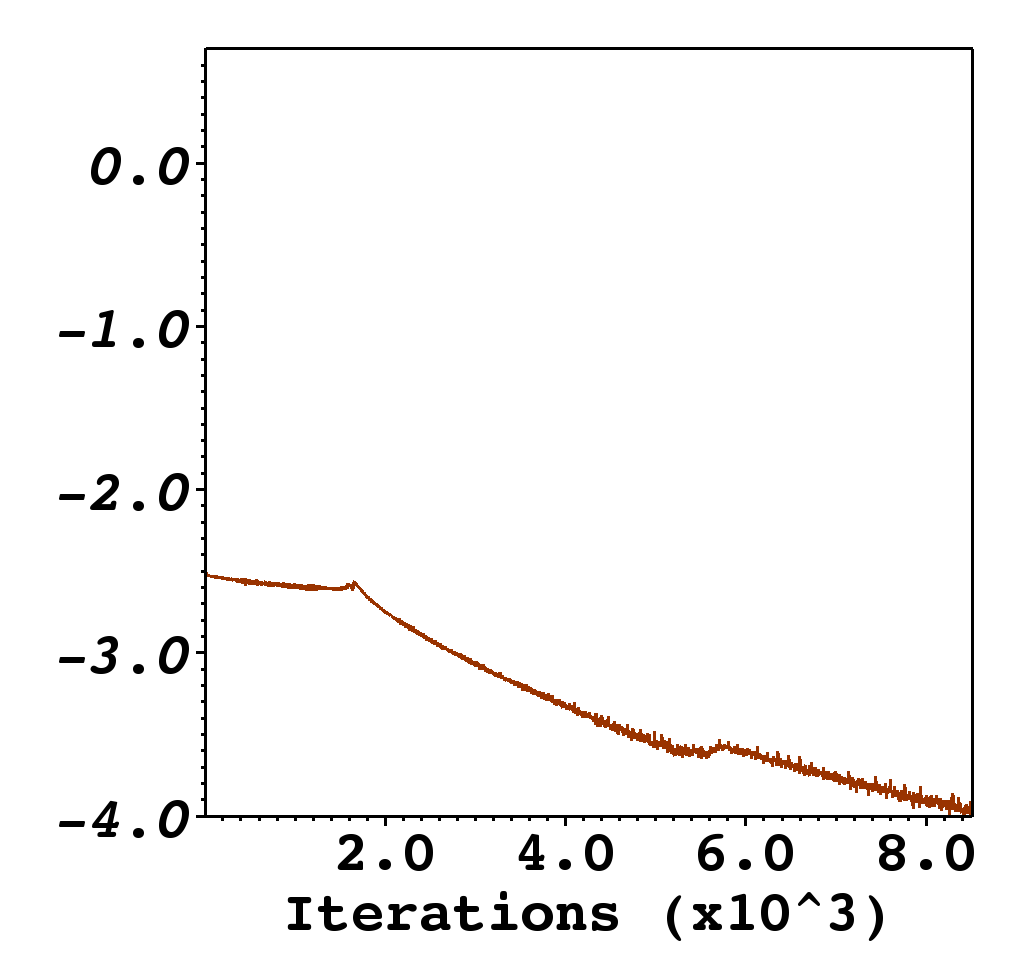} \quad
\includegraphics[trim={1.cm 0.cm 1.cm 1.cm},clip,scale=0.115]{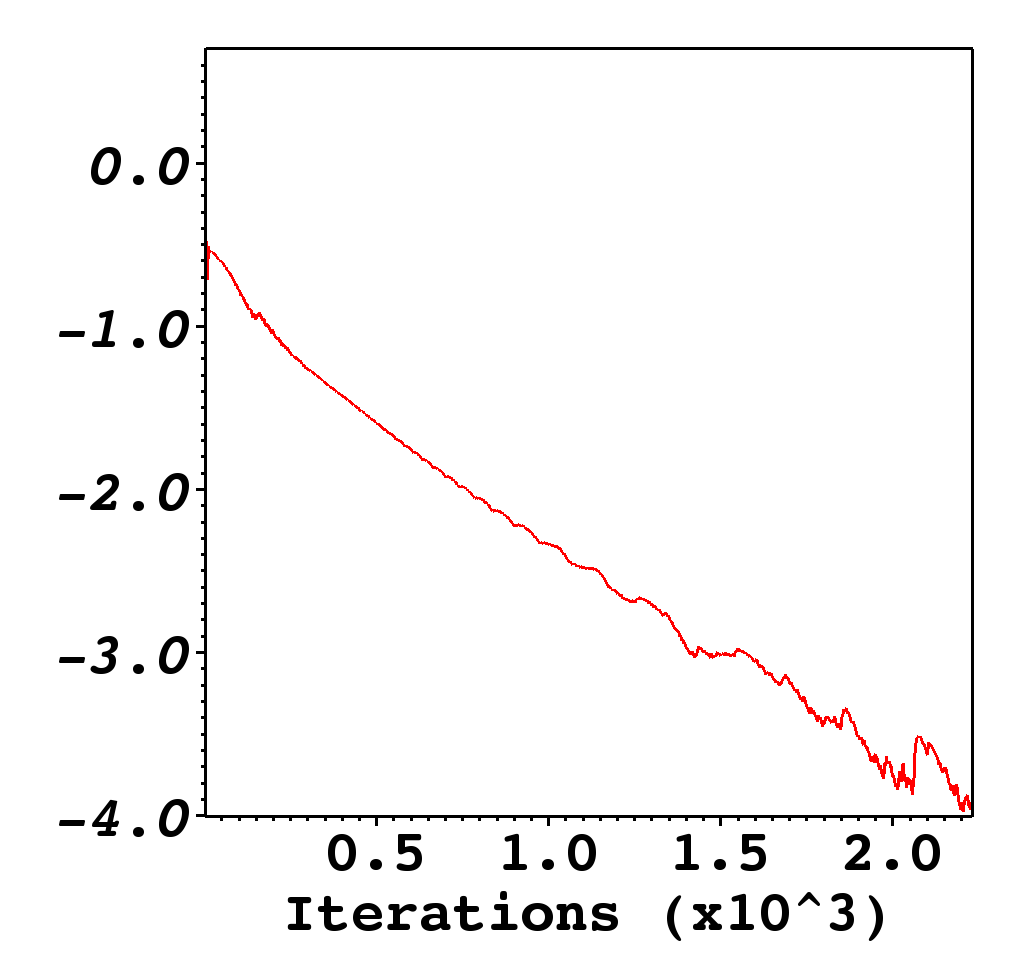}
\end{center}
\caption{Convergence of algorithms with criterion $\|\Div p + \lambda\|_{L^2} < 10^{-4}$ in case $\lambda=5$. From left to right, the first figure is in (ES), the second with (IS), and the last one is done with (ISS)}\label{compare3}
\end{figure}

We remark that we have to well pay for the cost of inverse Laplacian operator. So, the choice of solvers should be carefully considered. But, the positive side of semi-implicit scheme is to reduce globally many iterations of the iterative process. When the projections on convex sets become more expensive, this reduction of iteration will evidently save the computational time. At the moment, semi-implicit scheme is the most efficient, see Table \ref{table1}, Figure \ref{compare3} and Figure \ref{compare3method}.

\begin{table}
\begin{center}
\begin{tabular}{|c|c|c|c|c|c|c|}
\hline
\multicolumn{7}{|c|}{ $\| \Div p + \lambda \|_{L^2} < 10^{-4}$ with MPI in 6 processes  }\\
\hline
\multirow{2}{*}{$N$} & \multicolumn{3}{c|}{iterations} & \multicolumn{3}{c|}{time (seconds) }\\
\cline{2-7}
 & \bf ES & \bf IS & \bf ISS & \bf ES & \bf IS & \bf ISS \\
\hline
 101 & 9451 & 763 & 360 & 0.59 & 12.75 & 5.84 \\
 201 & 21647 & 1340 & 545 & 4.92 & 47.63 & 20.18\\
 301 & 34719 & 1857 & 733 & 17.96 & 92.01 & 39.93\\
 501 & 59438 & 2822 & 856 & 113.29 & 311.92 & 121.36\\
 801 & 98484 & 4072 & 1251 & 848.15 & 819.10 & 467.12\\
 1201 & 154107 & 5777 & 1596 & 3590.69 & 2625.47 & 1460.76\\
 1701 & 232793 & 7629 & 2038 & 11520.75 & 7642.35 & 3876.28\\
 1921 & 268999 & 8507 & 2230 & 17183.18 & 11174.52 & 5717.89\\
\hline
\end{tabular}
\end{center}
\vspace{1em}
\caption{Comparison of three methods in iterations and computational time in case $\lambda=5$. The stop criterion is $\| \Div p + \lambda \|_{L^2} < 10^{-4}$. They are implemented with MPI in 6 processes.}\label{table1}
\end{table}

\begin{figure}[ht]
\begin{center}
\includegraphics[scale=0.68]{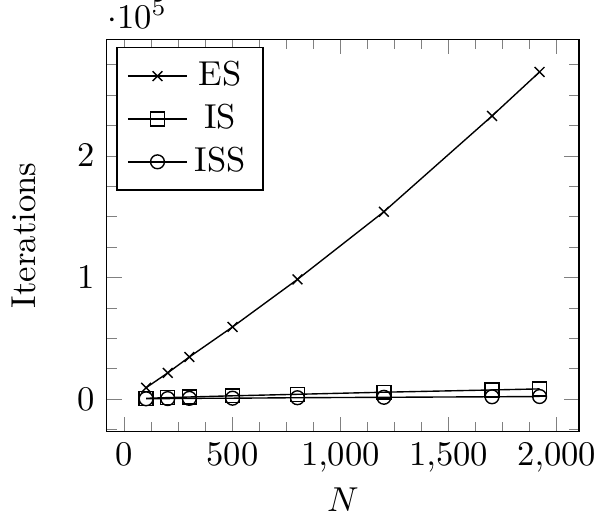}~
\includegraphics[scale=0.68]{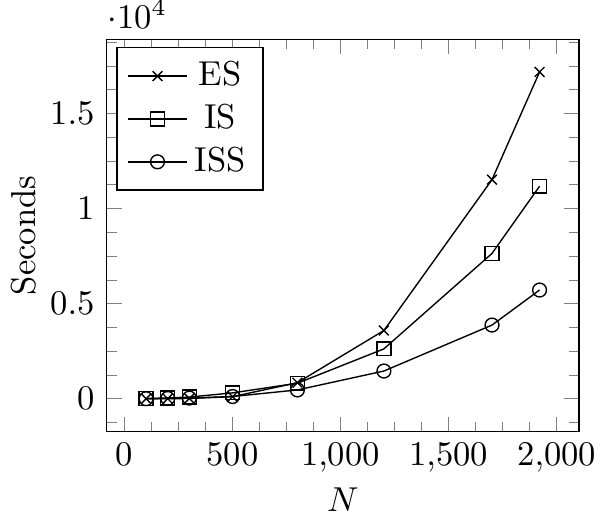}~
\includegraphics[scale=0.68]{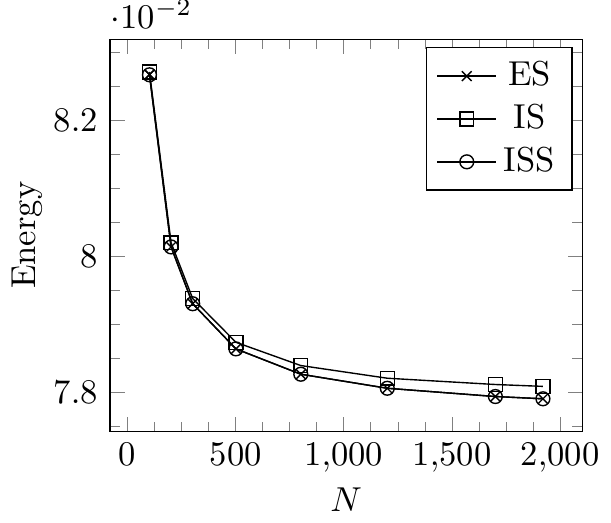}
\end{center}
\caption{Comparison of three methods in term of iterations, computational time, convergence of energy function in case $\lambda=5$.}\label{compare3method}
\end{figure}

\bigskip

\section{Conclusion}
The generalized explicit scheme for searching saddle points of Lagrangians of type \eqref{L} is convergent and widely applicable. The main contribution of this article is to propose  a semi-implicit extension of such an algorithm. It remains convergent under less restrictive constraint on numerical parameters.
Non differentiability can possibly occur in Lagrangians, and in this case, to fix it, we regularize derivatives of Lagrangians, with a smoothing parameter linked to the space discretization parameter.
 The semi-implicit scheme, coupled to a splitting method for the rapidly computed explicit part, provides a robust acceleration of the computational cost in comparison with the fully explicit one. 
 The number of iterations in order to reach a precise convergence is widely reduced as shown in Section \ref{appli}. Even if an iteration is more costly for the semi-implicit algorithm, the global computational cost is clearly reduced, specially for fine grids and provides accurate solutions. Furthermore, such an algorithm can reveal even more  performing than explicit scheme when Lipschitz constants functions $F$ and $G$ are lower and also if the Lagrangian contains large quadratic terms with respect to the $p$ variable. This algorithm has then been willingly tested on a stiff problem.
 In any case, the solver of a Laplacian type problem must be carefully chosen since it mainly contributes to the computational cost. Once set up, the heavier projections are, the more efficiency the semi-implicit scheme shows.

\bibliographystyle{siamplain}

\end{document}

%% file: ex_shared_fini.tex
\usepackage[T1]{fontenc}
\usepackage[utf8]{inputenc}
\usepackage{lmodern}
\usepackage{amsmath,amsfonts,amssymb}
\usepackage{multirow}

\usepackage{graphicx}
\graphicspath{{./../../Figure/finish/}}
\graphicspath{{./../../../../MINH-ARTICLE/Figure/finish/}}

\usepackage{pgfplotstable}
\usepackage{pgfplots}

\usepackage{tikz} \usetikzlibrary{calc,spy} \usetikzlibrary{patterns}\usetikzlibrary{intersections,through}
\usepackage{dsfont}
\usepackage{microtype}		
\usepackage{float}			
\usepackage[english]{babel}
\usepackage{hyperref} 
\hypersetup{colorlinks = true, linkcolor = blue}

\newtheorem{remark}[theorem]{Remark}

\def\e{\varepsilon}
\def\ov{\overline}

\DeclareMathOperator{\Div}{div}

\DeclareMathOperator{\im}{Im}
\DeclareMathOperator{\prox}{prox}
\DeclareMathOperator{\dom}{dom}

\newcommand{\LL}{\mathcal{L}}

\newcommand{\Ph}{\mathcal{P}_h}
\newcommand{\Pbh}{(\mathcal{P}_h)}
\newcommand{\Mh}{\mathcal{M}_h}
\newcommand{\Mbh}{(\mathcal{M}_h)}
\newcommand{\RR}{\mathbb{R}}

\newcommand{\tinm}[1]{\text{\; #1 \;}}

\usepackage{lipsum}
\usepackage{amsfonts}
\usepackage{graphicx}
\usepackage{epstopdf}
\usepackage{algorithmic}
\ifpdf
  \DeclareGraphicsExtensions{.eps,.pdf,.png,.jpg}
\else
  \DeclareGraphicsExtensions{.eps}
\fi

\numberwithin{theorem}{section}

\newcommand{\TheTitle}{A semi-implicit scheme based on Arrow-Hurwicz method for saddle point problems} 
\newcommand{\TheAuthors}{M. Phan, C. Galusinski}

\headers{Algorithm for primal dual problems}{\TheAuthors}

\title{{\TheTitle}\thanks{Submitted to the editors 01st december 2017.
}
}

\author{
  Minh  Phan\thanks{Laboratory IMATH (EA 2134), Universit\'e de Toulon, Toulon, France (\email{phan@univ-tln.fr}).}
  \and
  Cedric Galusinski\thanks{Laboratory IMATH (EA 2134), Universit\'e de Toulon, Toulon, France (\email{galusins@univ-tln.fr)}.}
}

\usepackage{amsopn}
